\documentclass[11pt]{amsart}
\usepackage{amsfonts}
\def\cal{\mathcal}
\def\Box{\hbox{$\sqcap \unskip \kern -6.5pt 
\sqcup$}}
\usepackage{graphicx}

\newtheorem{lemma}[subsection]{Lemma}
\newtheorem{corollary}[subsection]{Corollary}
\newenvironment{ulemma}{\par\noindent\textbf{Lemma}\,\,\em}{\rm}
\newenvironment{utheorem}{\par\noindent\textbf{Theorem}\,\,\em}{\rm}

\numberwithin{equation}{subsection}

\title{A connected  string of long thick and dominants}
\author{Mary Rees}
\begin{document}
\maketitle
\begin{abstract}
We prove that every Teichmuller geodesic of a finite type surface contains a string of intersecting {\em{long, thick and dominant}} segments, such that the distance between consecutive segments is bounded. This is key to obtaining some results about Teichm\"uller geodesics which mimic those for hyperbolic geodesics. These results have important applications to results about the geometry of hyperbolic three-manifolds.


\end{abstract}

\section{Introduction}

Teichm\"uller geodesics of finite type surfaces are very interesting objects. They are interesting in their own right, but their properties also have important applications. Two areas of applications come  to mind. One area is dynamics, especially  the dynamics of measured foliations, geodesic laminations and interval exchanges --- and, of course the Teichm\"uller geodesic flow itself, and the dynamics of the mapping class group action on various boundaries --- but the properties of Teichm\"uller geodesics do have dynamical implications at a much more  basic level. The other application is to topology and geometry. Teichm\"uller space and Teichm\"uller geodesics were used to study the topology and geometry of spaces of critically finite branched coverings in \cite{R1}. Teichm\"uller geodesics have also been used to study the geometry of hyperbolic three-manifolds, especially those with finitely generated fundamental groups. The links between these two general areas of applications -- to dynamics, and to  the topology and geometry of particular spaces --- are very strong. 

In some respects, properties of Teichm\"uller geodesics mimic properties of hyperbolic geodesics. Indeed, the Teichm\"uller space for the torus with at most one puncture, or for a sphere with four punctures, is the hyperbolic plane, and the geodesics are the usual hyperbolic geodesics. For all higher type surfaces, the Margulis decomposition of a hyperbolic surface into ``thick'' and ``thin'' parts induces an approximate product metric structure on the corresponding ``thin'' parts of Teichm\"uller space, which conflicts with properties related to hyperbolic space. There are various strategies for dealing with this. One is to consider another space altogether, which has stronger hyperbolic properties, such as the curve complex, which is Gromov hyperbolic \cite{M-M1,M-M2, Bow}. Another strategy is to work directly with Teichm\"uller space, and to project to suitable coordinates to get hyperbolic properties. It is this strategy which is followed with the theory of ``long thick and dominant'' pieces. This theory was originally developed in \cite{R1}. The development there was for punctured spheres --- simply because of the application for which it was intended. In fact, the theory works equally well for any finite type surface . A related theory, was developed by Rafi, in his thesis \cite{Raf}, with application to hyperbolic three-manifolds in mind.  This theory has since been used extensively for example \cite{Raf2, Raf3, Raf4}. 

The purpose of this article is to present a result (\ref{7.13}) about the long thick and dominant pieces on Teichm\"uller geodesics. Visually, the result is as follows. A Teichm\"uller geodesic is a path through hyperbolic surfaces of a fixed finite type. Each surface on the path has a Margulis decomposition into thick and thin subsurfaces. On each surface, there is either at least one ``thick'' piece, which is ``dominant'' -- which can loosely be taken to mean that the Teichm\"uller distance is moving on this surface at about the same rate as the distance on the whole surface --- or there is an annulus of large modulus on which the metric, induced by the quadratic differential associated to the quadratic differential, is approximately Euclidean -- in which case the metric distance on this annulus is, once again, moving at approximately the rate of the distance on the whole surface. In each case, we call this subsurface ``thick'' ---although an annulus of large modulus is in the thin part of the surface.  For suitable parameters, these thick subsurfaces persist for some time along the geodesic. If they persist for a time which is regarded as sufficiently long, we call them {\em{long, thick and dominant pieces}}. One of the basic results of \cite{R1} was that long, thick and dominant pieces do exist, in any sufficiently long Teichm\"uller geodesic.   If one were to make a three-dimensional model of the geodesic, out of wood, say, then these long thick and dominant pieces would look like chunky beads strung out along a necklace, with the cross-section of each bead in the shape of some subsurface. These beads could be very large. If a geodesic lies entirely in the thick part of Teichm\"uller space, then the model has a single bead, with cross-section in the shape of the whole surface, and running along the entire length of the necklace. But such geodesics are highly untypical. Precise quantification of the numbers and thicknesses of beads that one expects is not easy, but one certainly expects the number of beads to grow with the length of the geodesic, and that some of the beads will get more chunky with increasing length. In a broad sense, there is a huge and important literature related to this topic. The current purpose is concerned, more basically, with a property which holds for all geodesics. We will show that,  for any geodesic, or rather, for the necklace representation of any geodesic, there is a collection of beads, which cannot be moved more than a fixed length along the string without clashing against another bead in the collection. This means that the cross-sections of the  adjacent beads in the collection intersect. It is then a consequence of the long, thick and dominant properties that the cross-sections of any two beads in the collection intersect.

It is natural to expect that this property of the necklace implies a certain rigidity about manifestations of paths in Teichm\"uller space which might not be geodesic in the strict sense, but have some related properties: quasi-geodesic, perhaps. Such manifestations occur in any hyperbolic manifold of dimension three with finitely generated fundamental group, travelling out from the core to any one of the ends. The beads in this case can be smoothed out of pleated surfaces. Topologically, the same structure appears in a Teichmuller geodesic. Because of the immovability of the beads, more than a certain distance, the geometric structure on the two collections of  beads, up to bounded distortion, is the same: one in the three-manifold, one in the Teichm\"uller geodesic of surfaces. So a model of part of the three-manifold is obtained, in terms of the geometry of the Teichm\"uller geodesic. By an inductive procedure, the model can be extended to the whole manifold. Ultimately, the geometric structure of the hyperbolic three-manifold can be completely described, up to bounded distortion. A proof of the Ending Laminations Theorem follows. This proof can be found in \cite{R4}. 

The proof of \ref{7.13} given in \cite{R4} does have some errors, but the spirit of the proof is unchanged.
The proof is surprisingly difficult -- or, at least, it has elements which have not occurred in other related results, like the simple existence of long thick and dominants. Roughly speaking, rather than looking for a set of positive measure on a surface, we look for a small interval on a surface --- reducing to a point in the limit --- arising as an intersection of a decreasing sequence of intervals. By the same method we could produce a Cantor set of zero measure rather than a point, but not a set of positive measure. It is not clear why the proof is so difficult, but the reason could be significant, because the corresponding result for curve complexes --- the existence of a tight geodesic --- is comparatively trivial. 

\section{Teichmuller space.}\label{2}
\subsection{Very basic objects in surfaces}\label{2.0}

Unless otherwise stated, in this work, $S$ always denotes an oriented 
finite type surface without boundary, that is, obtained from a 
compact oriented surface by removing finitely many points, called 
{\em{punctures}}. One does not of course need an explicit realisation 
of $S$ as a compact minus finitely many points. One can simply take $S$ to be a finite type surface. Up to homeomorphism, $S$ is a compact minus finitely many points, with each end of $S$ mapped homeomorphically to a neighbourhood of the omitted point on the compact surface. A {\em{multicurve}} $\Gamma $ on 
$S$ is a union of simple closed nontrivial nonperipheral loops on $S$, 
which are isotopically distinct, and disjoint. A multicurve is 
{\em{maximal}} if it is not properly contained in any other 
multicurve. Of course, this simply means that the number of loops in 
the multicurve is $3g-3+b$, where $g$ is the genus of $S$ and $b$ the 
number of punctures. A {\em{gap}} is a connected open subsurface $\alpha $ of a 
given surface $S$ such that the topological boundary $\partial \alpha 
$ of $\alpha $ in $S$ is a multicurve. If $\Gamma $ is a multicurve on $S$, a  {\em{gap of $\Gamma $}}  is simply a component of $S\setminus (\cup \Gamma 
)$. If $\alpha $ is any gap, $\Gamma $ is a  {\em{multicurve in 
$\alpha $}} if it satisfies all the above conditions for a closed 
surface, and, in addition, $\cup \Gamma \subset \alpha $ and no loops 
in $\Gamma $ are homotopic to components of $\partial \alpha $. A 
positively oriented Dehn twist round a loop $\gamma $ on an oriented 
surface $S$ will always be denoted by $\tau _{\gamma }$.

 Let $\alpha _{i}\subset S$ be  a gap or loop for $i=1$, $2$, isotoped 
so 
that $\partial \alpha _{1}$ and $\partial \alpha _{2}$ have only 
essential intersections, or with $\alpha _{1}\subset \alpha _{2}$ 
if $\alpha _{1}$ is a loop which can be isotoped into $\alpha _{2}$. 
Then the {\em{convex hull}} $C(\alpha _{1},\alpha _{2})$ 
of $\alpha _{1}$ and $\alpha _{2}$ is the union of $\alpha _{1}\cup 
\alpha _{2}$ and any components of $S\setminus (\alpha _{1}\cup 
\alpha 
_{2})$ which are topological discs with at most one puncture. Then 
$C(\alpha _{1},\alpha _{2})$ is again a gap or a loop. The latter 
only 
occurs if $\alpha _{1}=\alpha _{2}$ is a loop. We are only interested 
in the convex hull up to isotopy, and it only 
depends on  $\alpha _{1}$ and $\alpha _{2}$ up 
to isotopy. It is so called because, if $\alpha _{i}$ is chosen to 
have geodesic boundary, and $\tilde {\alpha _{i}}$ denotes 
the preimage of $\alpha _{i}$ in the hyperbolic plane covering $S$, 
then up to isotopy $C(\alpha _{1},\alpha _{2})$ is the projection to $S$ of 
the convex hull of any component of $\tilde{\alpha _{1}}\cup 
\tilde{\alpha _{2}}$.

\subsection{Teichm\"uller space}\label{2.1} 
We consider Teichm\" uller space 
${\cal T}(S)$ of a surface $S$. If $\varphi _{i}:S\to 
S_{i}=\varphi _{i}(S)$ is an orientation-preserving homeomorphism, 
and $S_{i}$ is a complete hyperbolic surface with constant curvature 
$-1$, then we define the equivalence relation $\varphi _{1}\sim 
\varphi _{2}$ if and only if there is an orientation-preserving 
isometry $\sigma :S_{1}\to S_{2}$ such that $\sigma \circ \varphi 
_{1}$ is isotopic to $\varphi _{2}$. We define $[\varphi ]$  to be 
the equivalence class of $\varphi $, and ${\cal T}(S)$ to be the set 
of all such $[\varphi ]$, this being regarded as sufficient, since 
definition of a function includes definition of its domain. We shall 
often fix   a complete hyperbolic 
metric of constant curvature $-1$ on $S$ itself, which we shall also 
refer to as ``the'' Poincar\'e metric on $S$. 

Complete hyperbolic structure in dimension two is equivalent to complex 
structure, for any orientable surface $S$ of finite topological type and  
negative Euler characteristic, by the Riemann mapping theorem.  
So endowing such a surface $S$ with a complex 
structure defines an element of the Teichm\"uller space $\cal{T}(S)$. 
More generally, the Measurable Riemann Mapping Theorem  
implies that supplying a bounded measurable conformal structure for 
$S$ is enough to define an element of $\cal{T}(S)$, and indeed $\cal{T}(S)$ 
is often (perhaps usually) defined in this way.

\subsection{Teichm\"uller distance}\label{2.2}
We shall use $d$ to denote Teichm\"uller distance, so long as the 
Teichm\"uller space ${\cal T}(S)$ under consideration is regarded as 
clear. Moreover a metric $d$ will always be Teichm\"uller metric 
unless otherwise specified.
If more than one space is under consideration, we shall use 
$d_{S}$ to denote Teichm\"uller distance on ${\cal T}(S)$. The 
distance is defined as 
$$d([\varphi _{1}],[\varphi _{2}])={\rm{inf}}\{
{1\over 2} \log \Vert \chi \Vert 
_{qc}:[\chi \circ \varphi _{1}]=[\varphi _{2}]\} ,$$
where
$$\Vert \chi \Vert _{qc}=\Vert K(\chi )\Vert 
_{\infty }\vert ,\ \ K(\chi )(z)=\lambda (z)/\mu 
(z),$$
where $\lambda (z)^{2}\geq \mu 
(z)^{2}>0$ are the eigenvalues of $D\chi 
_{z}^{T}D\chi _{z}$, and $D\chi _{z}$ is the 
derivative of $\chi $ at $z$ (considered as a 
$2\times 2$ matrix) and $D\chi _{z}^{T}$ is its transpose. The 
infimum is achieved uniquely at a map $\chi $ which is given locally 
in terms of a unique quadratic mass 1 differential $q(z)dz^{2}$ on 
$\varphi 
_{1}(S)$, and its {\em{stretch}} $p(z)dz^{2}$ on $\varphi _{2}(S)$. 
The local coordinates are 
$$\zeta =x+iy=\int _{z_{0}}^{z}\sqrt{q(t)}dt,$$
$$\zeta '=\int _{z_{0}'}^{z'}\sqrt{p(t)}dt.$$
With respect to these local coordinates, 
$$\chi (\zeta )=\chi (x+iy)=\lambda x+i{y\over \lambda }=\zeta '.$$
So the distortion $K(\chi )(x+iy)=\lambda $ is constant. The singular 
foliations $x={\rm{constant}}$ and $y={\rm{constant}}$ on $\varphi 
(S)$ 
given locally by the coordinate $x+iy$ for $q(z)dz^{2}$ are known as 
the {\em{stable and unstable foliations for $q(z)dz^{2}$.}} We also 
say that $q(z)dz^{2}$ is the {\em{quadratic differential at $[\varphi 
_{1}]$ for $d([\varphi _{1}],[\varphi _{2}])$}}, and $p(z)dz^{2}$ is 
its {\em{stretch}} at $[\varphi _{2}]$.

\subsection{Thick and thin parts}\label{2.4}

Let $\varepsilon $ be any fixed Margulis constant for dimension two, 
that is, for any hyperbolic surface $S$, if $S_{<\varepsilon }$ is 
the 
set of points of $S$ through which there is a nontrivial closed loop 
of length 
$<\varepsilon $, then $S_{<\varepsilon }$ is a (possibly empty) union 
of cylinders with disjoint closures. Then 
$({\cal{T}}(S))_{<\varepsilon }$ is the set of $[\varphi ]$ for 
which $(\varphi (S))_{<\varepsilon }$ contains an least one 
nonperipheral cylnder. The complement of 
$({\cal{T}}(S))_{<\varepsilon }$ is $({\cal{T}}(S))_{\geq \varepsilon 
}$. 
We shall sometimes write simply ${\cal{T}}_{<\varepsilon }$ and 
${\cal{T}}_{\geq \varepsilon }$ if it is clear from the context which surface
is 
meant. We shall also write ${\cal{T}}(\gamma ,\varepsilon )$ for the 
set 
of $[\varphi ]$ such that $(\varphi (S))_{<\varepsilon }$ contains a 
loop homotopic to $\varphi (\gamma )$. If $\Gamma $ is a set of 
loops, we write 
$${\cal{T}}(\Gamma ,\varepsilon )=\cup \{ {\cal{T}}(\gamma 
,\varepsilon ):\gamma \in \Gamma \} .$$

\subsection{Length and the interpretation of Teichm\"uller 
distance}\label{2.3}
We fix a surface $S$. It will sometimes be convenient to fix a 
hyperbolic metric on $S$, in which case we shall use $\vert \gamma 
\vert $ to denote length of a geodesic path $\gamma $ with respect to 
this metric. With abuse of notation, for 
$[\varphi ]\in {\cal{T}}(S)$ and a
nontrivial nonperipheral closed 
loop $\gamma $ on $S$, we write $\vert \varphi (\gamma )\vert $ for 
the length, with respect to the Poincar\'e metric on the hyperbolic 
surface $\varphi (S)$, of the geodesic homotopic to $\varphi 
(\gamma )$. We write $\vert \varphi (\gamma )\vert '$ for a 
modification of this, obtained as follows. We change the metric in 
$\varepsilon _{0}$-Margulis tube of $\varphi (S)$, for some fixed 
Margulis constant $\varepsilon _{0}$, to the Euclidean metric for 
this 
complex structure in the $\varepsilon _{0}/2$-Margulis tube, so that 
the loop round the annulus is 
length $\sqrt{\vert \varphi (\gamma )\vert }$, and a convex-linear 
combination with the Poincar\'e metric between the $\varepsilon 
_{0}$-Margulis tubes and $\varepsilon _{0}/2$-Margulis tubes. Then we 
take $\vert \varphi (\gamma )\vert '$ 
to be 
the length of the geodesic isotopic to $\varphi (\gamma)$ with 
respect 
to 
this modified metric. If the geodesic homotopic to $\varphi 
(\gamma )$ does not intersect any Margulis tube, then, of course, 
$\vert \varphi (\gamma )\vert =\vert \varphi (\gamma )\vert '$. 
Then for a constant $C$ depending only on $S$ and $\varepsilon _{0}$.
\begin{equation}\label{2.3.1}\vert {\rm{Max}}\{ \vert \log  \vert 
\varphi 
_{2}(\gamma )\vert '
-\log \vert \varphi _{1}(\gamma )\vert '\vert :\gamma \in \Gamma \} 
-d([\varphi _{1}],[\varphi _{2}])\vert \leq C.\end{equation}
Here, $\Gamma $ can be taken to be the set of all simple closed 
nonperipheral 
nontrivial closed loops on $S$. This estimate on Teichm\"uller 
distance derives from the fact that $\vert \varphi (\gamma )\vert '$ 
is inversely proportional to the largest possible square root of 
modulus of an embedded annulus in 
$S$ 
homotopic to $\varphi (\gamma )$. See also 14.3, 14.4 and 14.7 of 
\cite{R1} 
(although the square root of modulus was mistakenly left out of 
\cite{R1}) 
but this estimate appears in other places, for example \cite{M-M2}.
We can simply take $\Gamma $ to be any set of simple 
closed nontrivial nonperipheral loops on $S$ such that   
 that every component of $S\setminus (\cup \Gamma )$ is a 
topological disc with at most one puncture. We shall call such a loop 
set {\em{cell-cutting}}

\subsection{ Projections to subsurface Teichm\"uller 
spaces}\label{2.5}

For any gap $\alpha \subset S$, we define a 
topological surface $S(\alpha )$ without boundary and a continuous 
map $\pi 
_{\alpha }:{\cal{T}}(S)\to {\cal{T}}(S(\alpha ))$. We define $\varphi 
_{\alpha }(S(\alpha ))$ by defining its conformal structure. After 
isotopy of $\varphi $, we can assume that all the components of 
$\varphi (\partial \alpha )$ are geodesic. We now write 
$\overline{\varphi (\alpha )}$ for the compactification of $\varphi 
(\alpha )$ obtained by cutting along $\varphi (\partial \alpha )$ and 
adding boundary components, each one isometric to some component of 
$\varphi (\partial \alpha)$. Then we form the Riemann 
surface $\varphi 
_{\alpha }(S(\alpha ))$ by attaching a once-punctured disc 
$\{ z:0<\vert z\vert \leq 1\} $ 
to $\overline{\varphi (\alpha )}$ along each component of $\varphi 
(\partial 
\alpha )$, taking the 
attaching map to have constant derivative with respect to length on 
the geodesics $\varphi (\partial \alpha )$ and length on the unit 
circle. Then we define $\varphi _{\alpha }=\varphi $ on $\alpha $ and 
then extend the map homeomorphically across each of the punctured 
discs. 
Then $[\varphi _{\alpha }]$ is a well-defined element of ${\cal 
T}(S(\alpha ))$.

Now let $\alpha $ be a nontrivial nonperipheral simple closed loop. 
Fix an 
orientation on $\alpha $ Then we define
$$S(\alpha )=\overline{\mathbb C}\setminus\{ \pm 2,\pm {1\over 2}\} 
.$$
Now we define an element $[\varphi _{\alpha }]=\pi _{\alpha 
}([\varphi ])\in \cal{T}(S(\alpha ))$, 
for each $[\varphi ]\in \cal{T}(S)$, as follows. Fix a Margulis 
constant $\varepsilon $. If $|\varphi (\alpha )|\leq \varepsilon $, 
let $A$ be the closed $\varepsilon $-Margulis tube in $\varphi (S)$ 
homotopic 
to $\varphi (\alpha )$. If $\vert \varphi (\alpha )\vert 
>\varepsilon $, let $A$ be the closed $\eta $-neighbourhood of the 
geodesic homotopic to $\varphi (\alpha )$ where $\eta $ is chosen so 
that $A$ is an embedded annulus, and thus can be chosen bounded from 
$0$ if $\vert \varphi (\alpha )\vert $ is bounded above. Fix a simple 
closed geodesic $\beta (\alpha )$ which intersects $\alpha $ at most 
twice and at least once, depending on whether or not $\alpha $ 
separates $S$. We can assume after isotopy that
$\varphi (\alpha )$ and $\varphi (\beta (\alpha ))$ are both 
geodesic, and we fix a  point $x_{1}(\alpha )\in \varphi (\alpha \cap 
\beta (\alpha))$. We make a Riemann surface $S_{1}$ homeomorphic to 
the sphere, by attaching a unit disc to each component of $\partial 
A$, taking the attaching maps to have constant derivative with 
respect to length.  Then we define $\varphi _{\alpha }$ to map 
$\overline{\mathbb C}$ to $S_{1}$ by mapping $\{ z:\vert z\vert =1\} 
$ 
to $\varphi (\alpha )$, $1$ to $x_{1}(\alpha )$, $\{ z:{1\over 
2}\leq \vert z\vert \leq 2\} $ to $A$ and $\{ z\in \mathbb R:{1\over 
2}\leq z\leq 2\} $ to the component of $\varphi (\beta (\alpha 
))\cap A$ containing $\alpha $. Then $\varphi _{\alpha }(S(\alpha ))$ 
is a four-times punctured sphere and so we have an element 
$[\varphi _{\alpha }]\in \cal{T}(S(\alpha ))$. Now the Teichm\"uller 
space $\cal{T}(S(\alpha ))$ is isometric to the upper half plane 
$H^{2}$ with metric ${1\over 2}d_{P}$, where $d_{P}$ denotes the 
Poincar\'e metric $\dfrac{dx^{2}+dy^{2}}{y^2}$. This is well-known. We now 
give an identification. 
Let  $n_{\alpha }([\varphi ])=n_{\alpha ,\beta (\alpha )}([\varphi 
])$ be the integer assigning 
the minimum value to 
$$m\to \vert \varphi (\tau _{\alpha }^{m}(\beta (\alpha ))\vert .$$
If there is more than one such integer then we take the smallest one. 
There is a bound on the number of such integers of at most two 
consecutive ones. We 
see this as follows. Let $\ell $ be a geodesic in the hyperbolic 
plane 
and let $g$ be a M\"obius involution such that $g.\ell $ is disjoint 
from, and not asymptotic, to $\ell $, and such that the common 
perpendicular 
 geodesic 
segment from $\ell $ to $g.\ell $ meets them in points $x_{0}$, 
$g.x_{0}$, for some $x_{0}\in \ell $. Then the complete
geodesics  meeting both $\ell $ and 
$g.\ell $ and crossing them both at the same angle, are precisely 
those that pass through points $x$ and $g.x$ for some $x\in \ell $, 
and the hyperbolic length of the segment joining $x$ and $g.x$ 
increases strictly with the length between $x_{0}$ and $x$. This 
implies the essential uniqueness of $n$, as follows. We take $\ell $ 
to be a lift of $\varphi (\alpha)$ to the universal cover, and let 
$\ell _{1}$ be another lift of $\varphi (\alpha )$, such that some 
lift of $\varphi (\beta (\alpha ))$ has endpoints on $\ell $ and 
$\ell _{1}$. Then $g$ is determined by making $\ell _{1}=g.\ell $ for 
$g$ as above. But also $\ell _{1}=g_{2}.\ell $, where $g_{2}$ is the 
element of the covering group corresponding to $\varphi (\beta 
(\alpha ))$. We also  have an element $g_{1}$ of the covering group 
corresponding to $\varphi(\alpha )$, which preserves $\ell $ and 
orientation on $\ell $. Then $\vert \varphi (\tau _{\alpha }^{m}(\beta 
(\alpha )))\vert $ is the distance between $x$ and $g.x$ for the 
unique $x$ such that some lift of a loop freely homotopic to  
$\varphi (\tau _{\alpha }^{m}(\beta (\alpha )))$ has endpoints at $x$ 
and $g.x$. The endpoints are $g_{1}^{-m}.y$ and $g_{2}.y$ for 
$y$ such that $x=g_{1}^{-m}.y$. So $x$ is determined by the $y=y_{m}$ 
such that $g.x=g_{2}.g_{1}^{m}.x$. So then $d(x,x_{0})={1\over 
2}d(x_{0},g^{-1}g_{2}g_{1}^{m}.x_{0})$, which takes its minimum at 
either one, or two adjacent, values of $m$.

Then the isometric identification 
with $H^{2}$ can be chosen 
so that, if we use the identification to regard $\pi _{\alpha }$ as a 
map to $H^{2}$, 
\begin{equation}\label{2.5.1}\pi _{\alpha }([\varphi ])=n_{\alpha 
}([\varphi ])+i\vert \varphi 
(\alpha 
)\vert ^{-1}+O(1).\end{equation}

If $\alpha $ is either a gap or a loop we now define a semimetric 
$d_{\alpha }$ by
$$d_{\alpha }([\varphi _{1}],[\varphi _{2}])=d_{S(\alpha )}(\pi 
_{\alpha }([\varphi _{1}]),\pi _{\alpha }([\varphi _{2}])).$$

\section{Teichm\"uller geodesics: long thick and dominant 
definitions.}\label{5}

In this section we explain and expand some of the ideas of long thick 
and dominant (ltd) segments of geodesics in Teichm\"uller space
$\cal{T}(S)$ which were used in \cite{R1}. The theory of \cite{R1} 
was 
explicitly for marked spheres only, because of the application in 
mind, but in fact the theory works, without adjustment, for any finite 
type surface, given that projections $\pi _{\alpha }$ to smaller 
Teichm\"uller spaces $\cal{T}(S(\alpha ))$ for subsurfaces $\alpha $ 
of $S$ have been 
defined in \ref{2.5}. For proofs, for the most part, we refer to 
\cite{R1}.
  The basic idea is to 
get into a position to apply arguments which work along geodesics 
which never enter the thin part of Teichm\"uller space, by projecting 
to suitable subsurfaces $\alpha $ using the projections $\pi 
_{\alpha }$ of \ref{2.5}.
We use  the basic notation and theory of Teichm\"uller space 
$\cal{T}(S)$ from Section \ref{2}.

\subsection{Good position.}\label{5.1} 

Let $[\varphi ]\in 
\cal{T}(S)$. Let $q(z)dz^{2}$ be a quadratic 
differential on $\varphi (S)$. All quadratic differentials, as in 
\ref{2.2}, will be of total mass $1$. Let $\gamma $ be a  
nontrivial nonperipheral simple closed loop on $S$. 
Then there is a limit of isotopies of 
$\varphi (\gamma )$ to 
{\em{good position}} with respect to $q(z)dz^{2}$, with the limit 
possibly passing through some punctures. If $\gamma $ is the 
isotopy limit, then either $\gamma $ is at constant angle to 
the stable and unstable foliations of $q(z)dz^{2}$, or is a union of 
segments between singularities of $q(z)dz^{2}$ which are at constant 
angle to the stable and unstable foliations, with angle $\geq \pi $ 
between any two consecutive segments at a singularity, unless it is a 
puncture.  An equivalent statement is that $\gamma $ is a geodesic with respect to 
the singular Euclidean metric $\vert q(z)\vert d\vert z\vert ^{2}$. 
If two good positions do not coincide, then they bound an open 
annulus 
in $\varphi (S)$ which contains no singularities of $q(z)dz^{2}$.
See also 14.5 of \cite{R1}. 

 The {\em{q-d length }} $\vert \varphi (\gamma )\vert _{q}$ is length with 
respect to the quadratic differential metric for any homotopy 
representative in good position. (See 14.5 of \cite{R1}.) 
We continue, as in Section \ref{2}, to use $\vert \varphi (\gamma 
)\vert $ 
to denote the hyperbolic, or Poincar\'e, length on $\varphi (S)$ of 
the geodesic on $\varphi (S)$ homotopic to $\varphi (\gamma )$. 
If $[\varphi ]\in \cal{T}_{\geq 
\varepsilon }$ then there is a constant $C(\varepsilon )>0$ such that 
for all nontrivial nonperipheral closed loops $\gamma $,
$${1\over C(\varepsilon )}\leq {\vert \varphi (\gamma )\vert 
_{q}\over \vert \varphi (\gamma )\vert }\leq C(\varepsilon ).$$
We also define $\vert \varphi (\gamma )\vert _{q,+}$ to be the 
integral of the norm of the projection of the derivative of $\varphi 
(\gamma )$ to the tangent space of the unstable foliation of 
$q(z)dz^{2}$, and similarly for $\vert \varphi (\gamma )\vert 
_{q,-}$. So these are both majorised by $\vert \varphi (\gamma 
)\vert 
_{q}$, which is, in turn, majorised by their sum.

\subsection{Area.}\label{5.2}

The following definitions come from 9.4 of \cite{R1}. 
For any essential nonannulus 
subsurface  $\alpha \subset S$, $a(\alpha 
,q)$ is the area with respect to $q(z)dz^{2}$ of $\varphi (\alpha )$ 
where $\varphi (\partial \alpha )$ is in good position and bounds the 
smallest area possible subject to this restriction. If $\alpha $ is a 
loop at $x$ then $a(\alpha ,q)$ is the smallest possible area of an 
annulus of modulus $1$ and homotopic to $\varphi (\alpha )$. We are 
only interested in this quantity up to a bounded multiplicative 
constant. It is boundedly proportional to $\vert\varphi 
(\alpha 
)\vert _{q}^{2}$ whenever $\varphi (\alpha )$ is in good position, 
and $\vert \varphi (\alpha )\vert $ is bounded. 
We sometimes write $a(\alpha ,x)$ or even 
$a(\alpha )$ for $a(\alpha ,q)$, if it is clear from the context what 
is meant.

Generalising from \ref{5.1}, there is a constant $C(\varepsilon )$ 
such that, if 
$\varphi (\alpha )$ is homotopic to a component of $(\varphi 
(S))_{\geq \varepsilon }$, then for all nontrivial nonperipheral 
non-boundary-homotopic closed loops $\gamma \in \alpha $,
$${1\over C(\varepsilon )}\vert \varphi (\gamma )\vert \leq {\vert 
\varphi (\gamma )\vert _{q}\over \sqrt{a(\alpha ,q)+a(\partial \alpha 
,q)}}\leq C(\varepsilon )\vert \varphi (\gamma )\vert .$$

Now suppose that $\ell $ is a directed geodesic segment in 
$\cal{T}(S)$
containing 
$[\varphi ]$, and that $q(z)dz^{2}$ is the quadratic differential 
at $[\varphi ]$ for $d([\varphi ],[\psi ])$ for any $[\psi ]$ in the 
positive direction along $\ell $ from $[\varphi ]$ (see \ref{2.1}.) 
Let $p(z)dz^{2}$ be the stretch of $q(z)dz^{2}$ at $[\psi ]$, and 
let $\chi $ be the minimum distortion map with $[\chi \circ \varphi 
]=[\psi ]$. Then $\chi $ maps the $q$-area element to the $p$-area 
element.
Then $a(\alpha ,q)=a(\alpha ,p)$ if $\alpha $ is a gap, but  if 
$\alpha $ is a loop, $a(\alpha ,y)$ varies for $y\in \ell $.

If $\alpha $ is a gap, we define $a'(\alpha )=a(\alpha )$. Now suppose $\alpha $ is a loop.  We define 
$a'(\alpha ,[\varphi ], q)$ (or simply $a'(\alpha )$ if the context 
is 
clear) to be the $q$-area of the largest modulus annulus (possibly 
degenerate) homotopic to $\varphi (\alpha )$ and with boundary 
components in good position for $q(z)dz^{2}$. Then in both cases, gap and loop,  
$a'(\alpha )$ is 
constant along the geodesic determined by $q(z)dz^{2}$.

\subsection{The long thick and dominant definition}\label{5.3}

Now we fix parameter functions $\Delta $, $r$, $s:(0,\infty )\to 
(0,\infty )$ and a constant $K_{0}$. 

Let $\alpha $ be a gap. Let $\ell $ be a geodesic segment.
We say that $\alpha $ is {\em{long, $\nu $-thick and dominant at }} 
$x$ 
(for $\ell $, and with respect to $(\Delta ,r,s)$) if $x\in \ell $ is the 
centre of a segment $\ell _{1}$ in the geodesic extending $\ell $ of  
length $2\Delta (\nu )$ 
such that $\vert \psi (\gamma )\vert \geq \nu $ for all $[\psi ]\in 
\ell _{1}$ and nontrivial nonperipheral $\gamma \subset \alpha $ not 
homotopic to boundary components, but $\ell _{1}\subset {\cal 
T}(\partial 
\alpha ,r(\nu ))$ and $a(\partial \alpha ,y)\leq s(\nu )a(\alpha, y)$ 
for all $y\in \ell _{1}$. We shall then also say that $\alpha $ is 
long $\nu $-thick and dominant along $\ell _{1}$. See 15.3 of 
\cite{R1}. 

A loop $\alpha $ at 
$x$ is $K_{0}$-{\em{flat at}}  $x=[\varphi ]$ (for $\ell $) if 
$a'(\alpha )\geq K_{0}a(\alpha)$.  This was not quite the definition 
made in 
\cite{R1}  
where the context was restricted to $S$ being a punctured sphere, 
but the results actually worked for any finite type surface. The term 
arises because if $\alpha $ is $K_{0}$-flat then the metric $\vert 
q(z)\vert dz^{2}$ is equivalent to a Euclidean (flat) metric on an 
annulus homotopic to $\varphi (\alpha )$ of modulus $K_{0}-O(1)$.  
For 
fixed $K_{0}$ we may simply 
say 
{\em{flat}} 
rather than $K_{0}$-flat. 

In future, we shall often refer to 
parameter 
functions as quadruples of the form $(\Delta ,r,s,K_{0})$. For a fixed quadruple, we shall refer to $(\alpha ,\ell )$ as {\em{ltd}} if either $\alpha $ is a gap which is long,  $\nu $- thick and dominant along $\ell $ with respect to these parameter functions, or  $\alpha $ is a loop which is $K_0 $-flat along $\ell $.

If $(\alpha ,\ell )$ is ltd, and $\alpha $ is a gap, then $d_{\alpha }(x,y)$ 
is very close to $d(x,y)$ for all $x$, $y\in \ell $. This is a 
consequence of the results of Section 11 of \cite{R1}. For us, here, the fact that the two quantities  differ by some additive constant is sufficient motivation. 
It 
is also probably worth noting (again by the results of Chapter 11 of 
\cite{R1}) that if $[\varphi ]\in \ell $ and
$\pi _{\alpha }([\varphi ])=[\varphi _{\alpha }]$, then $\varphi 
_{\alpha }(S(\alpha ))$ and the component $S(\alpha ,r(\nu ),[\varphi 
])$ of 
$(\varphi (S))_{\geq 
r(\nu )}$ homotopic to $\varphi (\alpha )$ are isometrically very 
close, 
except in small neighbourhoods of some punctures, and the quadratic 
differentials $q(z)dz^{2}/a(\alpha )$ at $[\varphi ]$ for $d([\varphi 
],[\psi ])$ ($[\psi ]\in \ell$) and the quadratic differential 
$q_{\alpha }(z)dz^{2}$ at $[\varphi _{\alpha }]$ for $d_{\alpha 
}([\varphi ],[\psi ])$, are very close. We refer to \cite{R1} for the main applications of ltd's, including a thin triangle result.

\subsection{}\label{5.4}
We now show that ltd's exist, in some abundance.
This was the content of the first basic result about ltd's  in 15.4 of \cite{R1}, which was stated only for $S$ being 
a punctured sphere, but the 
proof worked for a general finite type surface. 

\begin{ulemma} For some $\nu _{0}$ and $\Delta _{0}$ 
depending 
only on $(\Delta ,r,s,K_{0})$ (and the topological type of $S$), 
  the following holds. Any 
geodesic segment $\ell $ of length $\geq \Delta _{0}$ contains a 
segment $\ell '$ for which there is $\alpha $ such that one of the following holds.
\begin{itemize}
\item[.] 
$\alpha $ is a gap  which is long $\nu $-thick and dominant along 
$\ell '$ for some $\nu \geq \nu _{0}$ and $a(\alpha 
)\geq 1/(-2\chi (S)+1)=c(S)$ (where $\chi $ denotes Euler 
characteristic.
\item[.] $\alpha $  is a $K_{0}$-flat loop along $\ell '$ and $a(\alpha )\ge c_0$, where $c_0$ depends only on the topological type of $S$ and the ltd parameter functions
\end{itemize}
More generally,  there 
is $s_{0}$ depending 
only on $(\Delta ,r,s,K_{0})$ (and the topological type of $S$) 
such that, whenever $\omega \subset S$ is such that 
$a(\partial \omega )\leq s_{0}a(\omega )$, then we can find $\alpha $ 
as above with $\alpha \subset \omega $ and 
$a(\alpha )\geq c_0a(\omega)$.
\end{ulemma}

\begin{proof} (See also 15.4 of \cite{R1}.) We consider the 
case $\omega =S$.
Write $r_{1}(\nu )=e^{-\Delta (\nu )}r(\nu )$. 
Let $g=-2\chi (S)\ge 2$  and let
$r_{1}^{g}$ denote the $g$-fold iterate. We then take 
$$\nu _{0}=r_{1}^{g}(\varepsilon _{0})$$ 
for a fixed Margulis constant 
$\varepsilon _{0}$ and we define 
$$\Delta _{0}=2\sum _{j=1}^{g}\Delta (r_{1}^{j}(\varepsilon 
_{0})).$$ 
Then for some $j\leq g$, we can find $\nu =r_{1}^{j}(\varepsilon 
_{0})$ and $[\varphi ]\in \ell $ such that the segment $\ell '$ of 
length 
$2\Delta (\nu )$ centred on $y=[\varphi ]$ is contained in $\ell $, 
and such that 
for   any nontrivial nonperipheral loop 
$\gamma $, either $\vert \varphi '(\gamma )\vert \geq \nu $ for all 
$[\varphi ']\in \ell '$, or $\vert 
\varphi (\gamma )\vert \leq r_{1}(\nu )$ --- in which case 
$\vert 
\varphi '(\gamma )\vert \leq r(\nu )$ for all $[\varphi ']\in \ell 
'$. 
For any loop $\gamma $ with $\vert \varphi (\gamma )\vert 
<\varepsilon _{0}$, 
if $\beta $ is a gap such that $\gamma \subset \partial \beta $ and 
there 
is a  component of $(\varphi (S))_{\geq 
\varepsilon _{0} }$ homotopic to $\varphi (\beta )$ and separated 
from the flat 
annulus homotopic to $\varphi (\gamma )$ by an annulus of modulus 
$\Delta _{1}$, we have, since every zero of $q(z)dz^{2}$ has order at 
most $2g$, for a constant $C_{1}$ depending only on the topological 
type of $S$,
\begin{equation}\label{5.4.1}C_{1}^{-1}a(\gamma ,[\varphi ])e^{\Delta 
_{1}}
    \leq a(\beta )\leq 
C_{1}a(\gamma ,[\varphi ])e^{(2g+2)\Delta _{1}}.\end{equation}

Now let $\beta $ be a subsurface such that $\varphi (\beta )$ is 
homotopic to a component $S(\beta ,\nu )$ of 
$(\varphi (S))_{\geq \nu }$ or $(\varphi (S))_{\geq \nu }$ such that $a'(\beta )+a'(\partial \beta )$ is  of maximal area, where $\beta =\partial \beta $ if $\beta $ is a loop. This means that 
$$a'(\beta )+a'(\partial \beta )\ge \frac{1}{3g},$$
with $1/3g$ replaced by $1/g$, if $\beta $ is a gap.  Then by (\ref{5.4.1}), 
we have a bound of $O(e^{(2g^{2}+2g)/\nu })$ 
on the ratio of areas of any two components of $(\varphi (S))_{\geq 
\varepsilon _{0} }$ in $S(\beta ,\nu )$. If there is a component $\gamma $ of $\partial \beta $ such that
$$a'(\gamma )\geq e^{-1/(9gr_1(\nu))}a'(\beta ),$$
then we take $\gamma =\alpha $, and $\gamma $ is $K_0$-flat at $[\varphi ]$, assuming that $r(\nu )$ is sufficiently small, for all $\nu $ given $K_0$, and the lower bound on $a'(\beta )$ gives the required lower bound of $a'(\gamma )\ge c_0$, for $c_0$ depending only on $\nu _0$, that is, only on $(\Delta ,r)$ and the topological type of $S$. Otherwise, we take $\beta =\alpha $, and we have, for all $y'\in \ell '$,
$$a(\partial \alpha ,y')\leq e^{\Delta 
(\nu)}e^{-1/(9gr(\nu))}a(\alpha 
).$$ 
Assuming $r(\nu )$ is sufficiently small given $s(\nu )$ and 
$\Delta (\nu )$, $\alpha $ is long $\nu $-thick and dominant along 
$\ell '$ for $(\Delta, r,s)$, and $a(\alpha )\geq 1/(g+1)$, as required.

The case $\omega =S$ is similar. We only need $s_{0}$ small enough 
for $a(\partial \omega )/a(\omega )$ to remain small along a 
sufficiently long segment of $\ell $.
\end{proof}

Because of this result, we can simplify our notation. So let $\nu 
_{0}$ 
be as above, given $(\Delta ,r,s,K_{0})$. We shall simply say 
$\alpha $ is {\em{ltd}} (at $x$, or along $\ell _{1}$) 
if either $\alpha $ is a gap and long $\nu $-thick and dominant for 
some $\nu \geq \nu _{0}$, or $\alpha $ is a loop and $K_{0}$-flat, at $x$ or along $\ell _1$. 
We shall also say that $(\alpha ,\ell _{1})$ is ltd, and, if we want to be more specific, we shall say that $(\alpha ,\ell _1)$ is ltd with respect to $(\Delta ,r,s,K_0,\nu _0)$.

\subsection{}\label{5.5}
We refer to Chapters 14 and 15 of \cite{R1} for a summary of all the 
results 
concerning ltd's, where, as 
already stated, the context is restricted to $S$ being a punctured 
sphere, but the results work for any finite type surface. The main 
points about ltd's are, firstly, that they are good coordinates, in 
which 
arguments which work in the thick part of Teichm\" uller space can be 
applied, and secondly that there is only bounded movement in the 
complement of ltd's. This second fact
is worth scrutiny. It is, at first 
sight, surprising. It is proved in 15.14 of \cite{R1}, which we now 
state, actually slightly corrected since short interior loops in 
$\alpha $ were forgotten in the statement there (although the proof 
given there does consider short interior loops) and slightly expanded 
in the case of $\alpha $ being a loop. 

\begin{ulemma}    Fix 
long thick and dominant parameter functions 
$(\Delta ,r,s,K_{0})$, and 
let $\nu _{0}>0$ also be given and sufficiently small. 
Then there exists  $L=L(\Delta ,r,s,K_{0},\nu _{0})$ such that the 
following 
holds. Let  $\ell $ be a geodesic segment and let $\ell _{1}\subset 
\ell $ and let  
$\beta  \subset S$ be a maximal subsurface up 
to homotopy  with the property that $\beta \times \ell _{1}$ is 
disjoint from $\alpha \times \ell '$ for  all ltd's  $(\alpha  ,\ell ')$ with respect to $(\Delta ,r,s,K_0,\nu _0)$. Suppose 
also that all components of 
$\partial \beta $ are 
nontrivial nonperipheral. Then $\beta $  is a disjoint union of gaps 
and loops $\beta _1$ such that the following hold.
\begin{equation}\label{5.5.1}\vert \varphi (\partial \beta _1)\vert 
\leq 
L{\rm{\ for\ all\ }}[\varphi ]\in \ell _1
.\end{equation} 
If $\beta _1$ is a gap, then for all $[\varphi ]$, 
$[\psi ]\in \ell _1$ and nontrivial nonperipheral 
non-boundary-parallel closed loops $\gamma $ in $\beta _1$,
\begin{equation}\label{5.5.2} L^{-1}\leq {\vert \varphi 
(\gamma 
)\vert '\over 
\vert \psi (\gamma )\vert '}\leq L,\end{equation}
\begin{equation}\label{5.5.5}
    \vert \varphi (\gamma )\vert \geq L^{-1}.
    \end{equation}
If $\beta _1$ is a loop, then for all $[\varphi ]$, 
$[\psi ]\in \ell _1$,
\begin{equation}\label{5.5.3} \vert {\rm{Re}}(\pi _{\beta _1 
}([\varphi ])-\pi _{\beta _1 }([\psi ]))\vert \leq L.\end {equation}

 Also if $\gamma $ is 
 in the interior of $\beta $, and 
$\ell _{1}=[[\varphi _{1}],[\varphi _{2}]]$, then given $\varepsilon 
_{1}>0$ there exists $\varepsilon _{2}>0$ depending only on 
$\varepsilon _{1}$ and $(\Delta ,r,s,K_0,\nu _0)$
such that
\begin{equation}\label{5.5.4}\begin{array}{l}
{\rm{If\ }}\vert \varphi (\gamma )\vert 
< \varepsilon _{2},{\rm{\ then}}\cr
{\rm{Min}}(\vert \varphi _{1}(\gamma )\vert ,\vert \varphi 
_{2}(\gamma )\vert )\leq \varepsilon _{1}{\rm{,\ and}}\cr 
{\rm{Max}}(\vert \varphi _{1}(\gamma )\vert ,\vert \varphi 
_{2}(\gamma )\vert )\leq  L.\cr
\end{array}\end{equation}

\end{ulemma} 

If (\ref{5.5.1}), and either (\ref{5.5.2}) and (\ref{5.5.5}), or 
(\ref{5.5.3}) 
hold for $(\beta _1,\ell _{1})$, depending on whether $\beta _1$ is a gap 
or a loop, we say that 
$(\beta _1,\ell 
_{1})$ is {\em{bounded}} (by $L$). Note that $L$ depends on the ltd 
parameter 
functions, and therefore is probably extremely large compared with 
$\Delta (\nu )$ for many values of $\nu $, perhaps even compared 
with $\Delta (\nu _{0})$. 

Here are some notes on the proof. For fuller details, see 5.14 of 
\cite{R1}. First of all, under the assumption that $\partial \beta 
$ satisfies the condition (\ref{5.5.1}), it is shown that $\beta $ 
is a union of $\beta _1$ satisfying (\ref{5.5.1}) to (\ref{5.5.5}). 
First, we show that (\ref{5.5.2}) holds for all $\gamma \subset 
\beta $ such that $\vert \varphi _{i}(\gamma )\vert $ is bounded 
from $0$ for $i=1$, $2$, and that (\ref{5.5.4}) holds for $\beta $.
This is done by breaking $\ell $ into three segments, with $a'(\beta 
)$ dominated by $\vert \varphi (\partial \beta )\vert _{q}^{2}$ on 
the two outer segments $\ell _{-}$, $\ell _{+}$, where $q(z)dz^{2}$ 
is 
the quadratic differential at $[\varphi]$ for $\ell $. The middle 
segment $[[\varphi _{-}],[\varphi _{+}]]$  has to be of bounded 
length by 
the last part of \ref{5.4}, since there are no ltd's in $\beta $ 
along $\ell $. Then  $\vert \varphi (\partial \beta )\vert _{q}$ is 
boundedly proportional to $\vert \varphi (\partial \beta )\vert 
_{q,-}$ along $\ell _{-}$, and to  $\vert \varphi (\partial \beta 
)\vert _{q,+}$ along $\ell _{+}$.  We can obtain (\ref{5.5.2}) along 
$\ell _{+}$, at least for a nontrivial $\beta ' \subset \beta $ 
for 
which  we can ``lock ''  loops $\varphi (\gamma )$, for which $\vert 
\varphi (\gamma )\vert $ is bounded,  along stable segments 
to $\varphi  (\partial \beta )$. If $\beta ' \neq \beta $ 
and 
$\gamma '\subset \partial \beta ' $ is in the interior of $\beta 
$, 
then either  
$\vert \varphi _{+}(\gamma ')\vert $ is small, or  $\vert \varphi 
_{+}(\gamma ')\vert _{q_{+}}$ is dominated by $\vert \varphi 
_{+}(\gamma 
')\vert _{q_{+},-}$, where $q_{+}(z)dz^{2}$ is the stretch of 
$q(z)dz^{2}$ at $[\varphi _{+}]$. In the case when $\vert \varphi 
_{+}(\gamma ')\vert $ is small, there is some first point $[\varphi 
_{++}]\in \ell _{+}$ for which  $\vert \varphi 
_{++}(\gamma ')\vert _{q_{++}}$ is dominated by $\vert \varphi 
_{++}(\gamma 
')\vert _{q_{++},-}$, where $q_{++}(z)dz^{2}$ is the stretch of 
$q(z)dz^{2}$ at $[\varphi _{++}]$. For this point,  $\vert \varphi 
_{++}(\gamma ')\vert $ is still small, and can be locked to a small 
segment of $\varphi _{++}(\partial \beta )$. This means that we can 
deduce that $\vert \varphi _{2}(\gamma ')\vert $ is small, giving 
(\ref{5.5.4}). So one proceeds by induction on the topological 
type of $\beta $,   obtaining (\ref{5.5.2}) and (\ref{5.5.4})
for $\beta $ from that for 
$\beta \setminus \beta '$. Then (\ref{5.5.4}) and (\ref{5.5.2}) 
imply that the set of loops with $\vert \varphi _{1}(\gamma )\vert 
<\varepsilon _{1}$ or $\vert \varphi _{2}(\gamma )\vert <\varepsilon 
_{1}$, for a sufficiently small $\varepsilon _{1}$, do not intersect 
transversally. This allows for a decomposition into sets $\beta _1$ 
satisfying (\ref{5.5.1}), (\ref{5.5.2}) and (\ref{5.5.5}).
 One then has to remove the hypothesis 
(\ref{5.5.1}) for $\partial \beta $. This is done by another 
induction, considering successive gaps and loops $\beta '$ disjoint 
from all ltd's along segments $\ell '$ of $\ell $, with $\vert 
\varphi (\partial \beta ')\vert \leq \varepsilon _{0}$ for $[\varphi 
]\in \ell '$, possibly with $\partial \beta '=\emptyset $. One then 
combines the segments and reduces the corresponding $\beta '$, 
either combining two at a time, or a whole succession together, if the 
$\beta '$ are the same along a succession of segments. In finitely many 
steps, one reaches $(\beta ,\ell )$ finding in the process that 
$\partial \beta $ does satisfy (\ref{5.5.1}).

As for showing that $\beta $ satisfies (\ref{5.5.3}), that follows 
from the 
following lemma --- which is proved in 15.13 of \cite{R1}, but not 
formally stated. Note that if $\beta $ is a loop, $a'(\beta 
,[\varphi ])$ is constant for $[\varphi ]$ in a geodesic segment 
$\ell $, but $a(\beta  ,q)$ is proportional to $\vert \varphi 
(\beta )\vert _{q}^{2}$ (if $q(z)dz^{2}$ is the quadratic 
differential at $[\varphi ]$ for $\ell $), which has at most one 
minimum on the geodesic segment and  otherwise increases or decreases 
exponetially with distance along the segment, depending on whether 
$\vert \varphi (\beta )\vert _{q}$ is boundedly proportional to 
$\vert \varphi (\beta )\vert _{q,+}$ or  $\vert \varphi (\beta 
)\vert _{q,-}$. 
So for any $K_{0}$, the set of $[\varphi ]\in \ell $ for which 
$a'(\beta )\geq K_{0}a(\beta ,[\varphi ])$ is a single segment, up 
to bounded distance. This motivates the following.  

\begin{lemma}\label{5.6} Given $K_{0}$, there is $C(K_{0})$ such that 
the following holds. Let $\ell $ be any geodesic segment. Suppose 
that  $\beta $ is a loop and $a'(\beta )\leq K_{0}a(\beta ,[\varphi ])$ for all $[\varphi 
]\in 
\ell $. Then for all $[\varphi ]$, $[\psi ]\in \ell $,
$$\vert {\rm{Re}}(\pi _{\beta }([\varphi ])-\pi _{\beta }([\psi 
]))\vert \leq C(K_{0}).$$
\end{lemma}

\noindent{\em{Proof.}}
The argument is basically given in 15.13 of 
\cite{R1}. Removing a  segment of length bounded in terms of $K_{0}$, 
$\varepsilon _{0}$ at one end, we obtain a reduced segment $\ell '$ 
such that  that $a'(\beta )\leq \varepsilon _{0}a(\beta ,[\varphi 
])$ for all $[\varphi ]\in \ell '$.
We use the 
quantity $n_{\beta }([\varphi ])$ of \ref{2.5}, which is 
${\rm{Re}}(\pi _{\beta  }([\varphi ]))+O(1)$ and is given to within 
length $O(1)$ 
by $m$ minimising $\vert \varphi (\tau _{\beta }^{m}(\zeta ))\vert 
$ for a fixed $\zeta $ crossing $\beta $ at most twice (or a bounded 
number of times). This is the same to within $O(1)$ as the $m$ 
minimising
$\vert \varphi (\tau _{\beta }^{m}(\zeta ))\vert 
_{q}$ for any quadratic differential $q(z)dz^{2}$. (To see this, 
note that the shortest paths, in the Poincar\'e metric, across a 
Euclidean annulus $\{ z:r<|z|<1\} $, are the restrictions of straight lines through the origin.)  
Assume without loss of generality that $\vert \varphi (\beta )\vert 
_{q}$ 
is boundedly proportional to $\vert \varphi (\beta )\vert _{q,+}$ 
for 
$[\varphi]\in \ell $, and $q(z)dz^{2}$ the quadratic differential at 
$[\varphi ]$ for $\ell $. 
The good positions of  $\varphi (\tau _{\beta }^{m}(\zeta ))$ for 
all 
$m$ are locked together along stable segments whose qd-length is 
short in comparison with $\vert \varphi (\beta )\vert _{q}$, if 
$\varepsilon _{0}$ is sufficiently small. So $n_{\beta }([\varphi 
])$ 
varies by $<1$ on $\ell $, and is thus constant on $\ell '$, if 
$\varepsilon _{0}$ is sufficiently small, and hence varies by at most 
$C(K_{0})$ on $\ell $.\Box

\subsection{A Chain of ltd's.}\label{7.13}

Now the result we are aiming for is the following.
The proof of this result is different in character from that of 
\ref{5.4}, being, essentially, a construction of a zero measure 
Cantor 
set, while \ref{5.4} obtained a set with a lower bound on area. This 
result is in any case more sophisticated, because it uses 
\ref{5.5} --- and hence also \ref{5.4} --- in the course of the 
proof. This result can be regarded as a parallel to the existence of 
a tight geodesic in the curve complex used by Minsky et al.. \cite{Min2,B-C-M}. For 
reasons which are not entirely clear to me, but which may be 
significant, this result appears to be much harder to prove.

\begin{utheorem} Fix long thick and dominant parameter 
functions and flat constant $(\Delta ,r,s,K_{0})$. 
Then there exist $\Delta _{0}$, $\delta _0>0$ and $\nu 
_{0}$  depending 
only on $(\Delta ,r,s,K_{0})$
and the topological type of $S$ such that the following holds. 
Let $[y_{0},y_{T}]=[[\varphi _{0}],[\varphi _{T}]]$ be any geodesic 
segment 
in 
${\cal T}(S)$ 
of length $T\geq \Delta _{0}$, parametrised by length. 
Then 
there exists a sequence $(\alpha _{i},\ell _{i})$ ($1\leq i\leq 
R_{0}$)   such that:
\begin{itemize}
\item $(\alpha _i,\ell _i)$ is ltd with respect to $(\Delta ,r,s,K_0)$, and $\nu $-thick for some $\nu \ge \nu _0$ if $\alpha _i$ is a gap;
\item  $(\alpha _{i},\ell _{i})<(\alpha _{i+1},\ell 
_{i+1})$ 
for $i<R_{0}$, where the ordering $<$ is as in \ref{7.3}, that is, $\ell _i$ ends before $\ell _{i+1}$ starts, and $\alpha _i\cap \alpha _{i+1}\ne \emptyset $;
\item  each segment of $[y_{0},y_{T}]$ of 
length $\Delta _{0}$ intersects some $\ell _{i}$;
\item $a'(\alpha _i)\ge \delta _0$.
\end{itemize}
 
\end{utheorem}

\section{The proof of Theorem \ref{7.13}.}\label{6}

\subsection{Idea of the proof.}\label{7.0}
Throughout this chapter, we let $\{ y_t:t\in \mathbb R\} $ be a fixed Teichm\"uller geodesic, containing the segment $[y_0,y_T]$ which is the subject of Theorem \ref{7.13}. Here, $t$ parametrises length with respect to the Teichm\"uller metric. We write 
$y_{t}=[\varphi _t]=[\chi _{t}\circ \varphi _{0}]$,  with $\chi _{t}$ minimising distortion, with quadratic differential $q_0(z)dz^2$ and $y_0$ and stretch $q_t(z)dz^2$ at $y_t$ (see also \ref{2.2}).  We 
shall 
write $\vert .\vert _{t}$ for $\vert .\vert _{q_{t}}$ and $\vert 
\cdot \vert _{t,+}$ and $| \cdot |_{t,-}$ for the unstable and stable lengths $\vert \cdot \vert _{q_{t},+}$ and $|\cdot |_{q_t,-}$ respectively. (See \ref{5.1} for definitions.)

We shall prove Theorem \ref{7.13} by showing that, if $\Delta _{0}$ is 
sufficiently large, for  some $t_1\in [0,\Delta _0]$ and for some segment of unstable segment $\zeta _1\subset \varphi _{0}(\alpha _1)$,
and some $\xi \in \zeta _1$, for every Teichm\"uller geodesic segment $\ell $ of 
length $\Delta _{0}$, along $[y_{0},y_{T}]$, there is an ltd $(\alpha _i
,\ell _i)$
with $\ell _i\subset \ell $ and $\xi\in \varphi 
_{0}(\alpha _i)$. It then follows that the intersection   of all such $\varphi _{0}(\alpha 
_i)$ is nonempty, and therefore any two such $\alpha _i$ intersect essentially.
It is not  the case that any sequence $(\alpha _{i},\ell _{i})$ 
for $i\leq j$ is extendable, and this is the main obstacle that we have to overcome.

All the ltd's $(\alpha _j,\ell _j)$ have to be viewed in terms of $\zeta _1\cap \varphi _{0}(\alpha _j)$. In particular, we need to choose $\alpha _{j+1}$ so that 
$$\varphi _{0}(\alpha _j)\cap \varphi _{0}(\alpha _{j+1})\cap \zeta _1\ne \emptyset .$$
We transfer this to showing that, for a suitable $t_j$ such that $|\varphi _{t_j}(\partial \alpha _j)|$ is bounded at $[\varphi _{t_j}]$, and for a given  $\zeta _j\subset \zeta _1\cap \varphi _0(\alpha _j)$,
$$\varphi _{t_j}(\alpha _j)\cap \varphi _{t_j}(\alpha _{j+1})\cap \chi _{t_j}(\zeta _j)\ne \emptyset .$$
Basically, we will  need to choose $\alpha _{j+1}$ intersecting $\alpha _j$ so that $\varphi _{t_j}(\alpha _{j+1})\cap \chi _{t_j}(\zeta _j)$ is in the complement of a ``bad subset'' of $\varphi _{t_j}(\alpha _j)\cap \chi _{t_j}(\zeta _j)$, where the ``bad subset'' contains all  $\varphi _{t_j}(\beta )\cap \chi _{t_j}(\zeta _j)$ such that $\beta $ is bounded along a suitably chosen sufficiently long Teichm\"uller geodesic segment  $\ell \subset [y_{t_j},y_{t_j+\Delta _0}]$. This requires showing that the  bad subset is sufficiently small in a useful sense. One thing that we can do, as we shall see, is to bound  $a'(\beta )$ for these $\beta $. We use $a'(\beta )$ rather than $a(\beta )$, as $a'(\beta )$ records the area of a subset of $\varphi _t(S)$ which is constant as $t$ varies. But this is not enough, because an area bound is clearly not enough to bound  length of intersection with an arc $\chi _{t_j}(\zeta _j)\subset \varphi _{t_j}(S)$. We need to bound the set of such $\varphi _{t_j}(\beta )\cap \chi _{t_j}(\zeta _j)$ within a set of a certain shape: a union of intervals which are sufficiently short, and bounded apart by intervals which are sufficiently long. It turns out that we can do this, provided we enlarge the bad set in a certain way. The method involves a careful comparison between Poincar\'e distance along the surfaces of the Teichm\"uller geodesic $[y_0,y_T]$, and the distance determined by the different quadratic differentials $q_t(z)dz^2$. In particular, we will use properties of the graph of the {\em qd-length function} $\log |\varphi _t(\partial \beta )|_t$. We will do this in \ref{7.20}.

But first we recall the partial order properties of ltd's, which are derived as follows. The whole of the theory of ltd gaps and loops is based on a simple 
dynamical lemma which quantifies density of leaves of 
the stable and unstable foliations of a quadratic differential. This 
is basically 15.11 of \cite{R1}. But the statement is slightly more general. 

\begin{lemma}\label{7.1}Let a deceasing function $\varepsilon :(0,\infty )\to (0,\infty )$ be given. Then there is  a  function $L:(0,\infty )\times (0,\infty )\to (0,\infty )$ which is decreasing in the first coordinate and increasing in the second, such that the following holds. Let $[\varphi ]\in {\cal{T}}(S)$ and let $q(z)dz^2$ be a quadratic differential at $[\varphi ]$. Let $\alpha $ be a gap with $\varphi (\partial \alpha )$ in good position, with $a(\alpha )=a$ and $|\varphi (\partial \alpha )|_q\le M\sqrt{a}$. Let $J\subset \varphi (\alpha\cup\partial \alpha )$ be a segment of stable foliation with $|J|_q\ge \delta \sqrt{a}$. Then either every unstable segment in $\varphi (\alpha \cup \partial \alpha )$ of length $\ge L(\delta ,M)\sqrt{a}$ intersects $J\cup \varphi (\partial \alpha )$ or there is a closed loop $\gamma \subset {\rm{int}}(\alpha )$ with $|\varphi (\gamma )|_q\le L\sqrt{a}$ for  some $L\le L(\delta ,M)$, and $|\varphi (\gamma )|_{q,-}\le \varepsilon (L)\sqrt{a}$. 

Similar statements hold with the role of stable and unstable reversed.
\end{lemma}
\begin{proof} By taking the oriented cover of the unstable foliation, we can assume without loss of generality that both the stable and unstable foliations are orientable. Then we fix a segment $J$ of stable foliation with $|J|_q\ge \delta \sqrt{a}$. We write $J=\cup _{i=1}^pJ_i$, and all unstable leaves starting from the interior of  $J_i$ in the positive direction either return to $J_i$ without hitting singularities, or cross $\varphi  (\alpha )$ without hitting singularities. Write $R_j$ for the trapezium with base $J_j$ formed in this manner. We number so that $|J_i|_q$ is decreasing in $i$. Then$|J_1|_q\ge \delta \sqrt{a}/p$ and hence all unstable segments in $R_1$ have unstable length $\le (p/\delta +M)\sqrt{a} $ if the segments cross $\varphi (\partial \alpha )$ --- using the bound $|\varphi (\partial \alpha )|_q\le M\sqrt{a}$ and one segment being of qd-length $\le p\sqrt{a}/\delta $. Similarly all unstable segments in $R_j$ have unstable length $\le a/|J_j|_q+M\sqrt{a}$.  Write $x_j=|J_j|_q/\sqrt{a}$. Either $\sum _{k>j}x_k\ge \varepsilon (1/x_j+M)$ -- in which case we also have $x_{j+1}\ge  \varepsilon (1/x_j+1/\delta )/p$ --- or there is a least $j$ such that $\sum _{k>j}x_k< \varepsilon (1/x_j+M)$. In the first case we can write $g(x)=\varepsilon (1/x+M)/p$ -- which is an increasing function of $x$ -- and we obtain $x_i\ge g^{i-1}(p/\delta +M )$ for all $i$, where $g^{i-1}$ is the $i$-fold composition. In the second case we obtain this for $i\le j$. So we obtain the result for $L(\delta ,M)=1/(g^{p-1}(p/\delta +M))$.
\end{proof}

\begin{corollary}\label{7.21}
Given $\delta >0$, the following holds for suitable 
ltd parameter functions $(\Delta ,r,s,K_{0})$ and for a suitable function $L(\delta 
,\nu )$. Let $\alpha $ be a gap which is long $\nu $-thick 
and dominant along a segment $\ell =[[\varphi _{1}],[\varphi _{2}]]$ 
and let $[\varphi ]\in \ell $  
with $d([\varphi ],[\varphi _{1}])\geq \Delta (\nu )$. Let 
$q(z)dz^{2}$ be the 
quadratic differential at $[\varphi ]$ for $d([\varphi ],[\varphi 
_{2}])$ with stable and unstable foliations $\cal{G}_{\pm }$. Let 
$a=a(\alpha,q)$. 

Then there is no segment of the unstable foliation of 
qd-length $\leq 2L(\nu ,\delta )\sqrt{a}$ with both ends on $\varphi (\partial 
\alpha )$, and every segment of the unstable foliation of qd-length 
$\geq L(\nu ,\delta )\sqrt{a}$ in $\varphi (\alpha )$ intersects 
every segment of stable foliation of length $\geq \delta \sqrt{a}$. 

Similar statements hold with the role of stable and unstable reversed.
\end{corollary}
\begin{proof} We apply the lemma with $\varepsilon (L)=C(\nu )L^{-1}$ for a suitable constant $C(\nu )$ relating the qd-metric and Poincar\'e metric, which ensures that if the second option $|\varphi (\gamma )|_q\le L\sqrt{a}$ and $|\varphi (\gamma )|_{q,-}\le \varepsilon (L)$ holds for $L$ bounded in terms of $\Delta (\nu )$ and $\gamma \subset \alpha $ then there is $[\psi ]\in [[\varphi _1],[\varphi ]]$ with $|\psi (\gamma )|<\nu $, contradicting $\nu $-thickness. Also if there is $\zeta \subset \alpha $ with endpoints on $\partial \alpha $ and not homotopic into the boundary such that $\varphi (\zeta )$ is a segment of unstable foliation and  $|\varphi (\zeta )|\le L(\nu ,\delta )$, then adding in arcs along $\partial \alpha $ we again obtain a loop $\gamma \subset \alpha $ and $[\psi ]\in [[\varphi _1],[\varphi ]]$ with $|\psi (\gamma )|<\nu $, which again gives a contradiction.\end{proof}

\subsection{Loops cut the surface into cells.}\label{7.2}
There are  two fairly simple, but key, results, both of 
which follow directly from \ref{7.1}. These properties 
are used several times in 
\cite{R1}, 
but may never be explicitly 
stated. The first may be reminiscent of the concept of tight 
geodesics 
in the curve complex developed by Masur and Minsky \cite{M-M2}, and 
the point may be that these occur ``naturally'' in Teichm\" uller 
space.

\begin{ulemma} Given $L>0$, there is a function $\Delta _{1}(\nu )$ 
depending only 
on the topological type of $S$, such that the following holds for 
suitable parameter functions $(\Delta ,r,s,K_{0})$ 
Let $\alpha $ be a gap which is long $\nu $-thick and dominant along 
$\ell$ for $(\Delta , r, s,K_{0})$, with $\Delta (\nu 
)\geq \Delta _{1}(\nu )$. Let $y_{1}=[\varphi _{1}]$, $y_{2}=[\varphi 
_{2}]\in 
\ell $ with $d(y_{1},y_{2})\geq \Delta _{1}(\nu )$. 
Let $\gamma _{i}\subset \alpha $ with $\vert \varphi _{i}(\gamma 
_{i})\vert \leq L$, $i=1$, $2$. 

Then $\alpha \setminus (\gamma 
_{1}\cup \gamma 
_{2})$ is a union of topological discs with at most one puncture and 
topological annuli parallel to the boundary. Furthermore, for a 
constant $C_{1}=C_{1}(L,\nu )$, 
$$\# (\gamma 
_{1}\cap \gamma _{2})\geq C_{1}\exp 
d(y_{1},y_{2}).$$\end{ulemma}

\begin{proof} Let $[\varphi ]$ be the midpoint of 
$[[\varphi _{1}],[\varphi _{2}]]$, and let 
$q(z)dz^{2}$ be the quadratic differential 
for $d([\varphi ],[\varphi _{2}])$ at $[\varphi ]$. As before, write $a=a(\alpha )$. 
Because $\vert \psi  (\gamma 
_{i})\vert \geq \nu $ for all $[\psi ]\in [[\varphi _{1}],[\varphi 
_{2}]]$, by \ref{5.2}, the good position  of $\varphi _{1}(\gamma 
_{1})$ satisfies 
$$\vert \varphi _{1}(\gamma _{1})\vert _{q,+}\geq C(L,\nu 
)\sqrt{a},$$
and similarly for $\vert \varphi _{2}(\gamma _{2})\vert _{q,-}$. So 
$$\vert \varphi (\gamma _{1})\vert _{ q,+}\geq C(L,\nu )e^{\Delta 
_{1}(\nu )/2}\sqrt{a},$$
$$\vert \varphi (\gamma _{2})\vert _{ q,-}\geq C(L,\nu )e^{\Delta 
_{1}(\nu )/2}\sqrt{a}.$$
Then \ref{7.1} implies that, given $\varepsilon $, if 
$\Delta _{1}(\nu )$ is large enough given $\varepsilon $, 
$\varphi (\gamma _{1})$ cuts every segment of stable 
foliation of $q(z)dz^{2}$ of qd-length $\geq \varepsilon 
\sqrt{a}$ and 
$\varphi (\gamma _{2})$ cuts every segment of unstable 
foliation of $q(z)dz^{2}$ of qd- length $\geq \varepsilon 
\sqrt{a}$. So 
components of $\varphi (\alpha )\setminus (\varphi (\gamma _{1})\cup 
\varphi (\gamma _{2}))$ have Poincar\'e diameter
$<\nu $ if $\Delta _{1}(\nu )$ is sufficiently large, and must be
topological discs with at most one puncture or boundary-parallel 
annuli.

The last statement also follows from \ref{7.1}. 
If $d(y_{1},y_{2})<\Delta 
_{1}(\nu )$, there is nothing to prove, so now assume that 
$d(y_{1},y_{2})\geq \Delta 
_{1}(\nu )$. It suffices to bound below the number of intersections 
of $\varphi (\gamma _{1})$  and $\varphi (\gamma _{1})$. Let $L(\nu 
,1)$ be as in \ref{7.1}, and assume without loss of generality that 
$L(\nu ,1)\geq 1$. Supppose that $\Delta _{1}(\nu )$ is large enough 
that  each of $\varphi (\gamma _{1})$ and $\varphi (\gamma _{2})$ 
contains 
at least one segment which is a qd-distance $\leq \sqrt{a}/L(\nu ,1)$ 
from 
unstable and stable segments, respectively, of  qd-length 
$\geq \sqrt{a}L(\nu ,1)$. Note that the number of singularities of 
the 
quadratic differential is bounded in terms of the topological type 
of $S$. So 
apart from length which is a  bounded multiple of $L(\nu 
,1)\sqrt{a}$, 
each of $\varphi (\gamma _{1})$ and $\varphi (\gamma _{2})$ is a 
union of such segments. Then applying \ref{7.1}, each such segment 
of $\varphi (\gamma _{1})$ intersects each such segment on 
$\varphi (\gamma _{2})$. So we obtain the result for 
$C_{1}=c_{0}L(\nu 
,1)^{-2}$, for $c_{0}$ depending only on the topological type of $S$.
\end{proof}

\subsection{A partial order on ltd $(\beta ,\ell )$.}\label{7.3}

The following consequence of \ref{7.1} is important for the implications of our main result, \ref{7.13}. It allows us to define a partial order on long thick and dominants $(\beta ,\ell )$. The meaning of this, in terms of our ``beads'' metaphor, is that as adjacent beads  representing $(\alpha _i,\ell _i)$ and $(\alpha _{i+1},\ell _{i+1})$ of the sequence of \ref{7.13} cannot slide past each other, no two distinct beads representing $(\alpha _i,\ell _i)$ and $(\alpha _j,\ell _j)$ can change positions on the string. (Of course, if they could do so, the metaphor would make no physical sense.)
\begin{ulemma} For $i=1$, $3$, let $y_{i}=[\psi _{i}]\in \ell _{i}$, 
and 
let $\beta _{i}$ be a subsurface of $S$ with $\vert \psi 
_{i}(\partial 
\beta _{i})\vert \leq L$. Let  ltd parameter 
functions be suitably chosen given $L$. 
Let  $\ell _{2}\subset [y_{1},y_{3}]$ and  let $\beta 
_{2}\cap \beta _{i}\not = \emptyset $ for both $i=1$, $3$, and let 
$\beta _{2}$ 
be ltd along $\ell _{2}$. Then $\beta _{1}\cap \beta _{3}\not = 
\emptyset $, and $\beta _{2}$ is in the convex hull of 
$\beta _{1}$ and $\beta _{3}$. \end{ulemma}
\begin{proof}
This is obvious unless both $\partial 
\beta _{1}$ and $\partial \beta _{3}$ intersect the interior of 
$\beta 
_{2}$. So now suppose that they both do this.
First suppose that $\beta _{2}$ is a gap and long, $\nu $-thick and 
dominant. Let $y_{2,1}=[\psi _{2,1}]$, $y_{2,3}=[\psi _{2,3}]\in \ell 
_{2}$ 
with $y_{2,i}$ separating $\ell 
_{i}$ from $y_{2}$, with $y_{2,i}$ distance $\geq {1\over 3}\Delta 
(\nu)$ from 
the ends of $\ell _{2}$ and from $y_{2}$. If $\beta _{2}$ is a loop, 
then we can take these distances to be $\geq {1\over 6}\log K_{0}$. 
For $[\psi ]\in[y_{-},y_{+}]$, let $\psi (\beta )$ denote the 
region bounded by $\psi (\partial 
\beta )$ and homotopic to $\psi (\beta )$, assuming $\psi (\partial 
\beta )$ is in good position with respect to the quadratic 
differential at $[\psi ]$ for $[y_{-},y_{+}]$.
Then if $\beta _{2}$ is a gap, 
$\psi _{2,1}(\partial \beta _{1}\cap \beta _{2})$ 
includes a union of segments  in approximately unstable direction, 
of Poincar\' e length bounded from $0$, and similarly for 
$\psi _{2,3}(\partial \beta _{3}\cap \beta _{2})$, with 
unstable replaced by stable. Then as in \ref{7.2}, 
 $\psi _{2}(\partial \beta _{3}\cap \beta _{2})$ 
and $\psi _{2}(\partial \beta _{1}\cap \beta _{2})$
cut $\psi _{2} (\beta _{2})$ into topological discs with at 
most one puncture and annuli parallel to the boundary. It follows 
that $\beta _{2}$ is contained in the convex hull of $\beta _{1}$ and 
$\beta _{3}$. If $\beta _{2}$ is a loop, it is simpler. We replace 
$\psi (\beta _{2})$ by the maximal flat annulus $S([\psi ])$ 
homotopic 
to $\psi (\beta _{2})$, for $[\psi ]\in \ell _{2}$. Then $\psi 
_{2}(\partial \beta _{1})\cap S([\psi _{2}])$ is in approximately 
the unstable direction and $\psi 
_{2}(\partial \beta _{3})\cap S([\psi _{2}])$ in approximately the 
stable direction. They both cross $S([\psi _{2}]$), so must intersect 
in a loop homotopic to $\psi _{2}(\beta _{2})$. \end{proof}

We  define $(\beta _{1},\ell _{1})<(\beta _{2},\ell _{2})$ 
if 
$\ell _{1}$ is to the left of $\ell _{2}$ (in some common geodesic 
segment) and $\beta _{1}\cap \beta _{2}\not = \emptyset $. We can 
make 
this definition for any segments in a larger common geodesic segment, 
and even for single points in a common geodesic segment. So in the 
same way we can define $(\beta _{1},y_{1})< (\beta _{2},\ell _{2})$ 
if $y_{1}$ is to the left of $\ell _{2}$, still with  
$\beta _{1}\cap \beta _{2}\not = \emptyset $, and so on. This 
ordering is transitive restricted to ltd's $(\beta _{i},\ell _{i})$
by the lemma. 

\subsection{The graph of the qd-length function.}\label{7.20}

One of the basic technical considerations in the study of 
Teichm\"uller geodesics, as is probably already apparent, is the 
difference between the qd- and Poincar\'e metrics. The two metrics 
are not globally Lipschitz equivalent. But they are Lipschitz equivalent, 
up  to scalar, on any thick part of a surface. The Lipschitz constant 
is bounded in terms of the topological type of the surface, but the 
scalar 
is completely uncontrollable. This should not be regarded as a 
problem. One simply has to look at ratios of lengths rather than at 
absolute lengths. 

We consider a fixed geodesic segment $[y_0,y_T]$, as in \ref{7.0} and \ref{7.13}. We use the notation of \ref{7.0}, so that $y_t=[\varphi _t]=[\chi _t\circ \varphi _0]$, and $q_t(z)dz^2$ is the quadratic differential at $y_t$ for $d(y_0,y_t)$: the stretch of $q_0(z)dz^2$ at $y_0$ (see \ref{2.2}). For 
any 
finite loop set $\gamma $, we define the {\em{qd-length function for $\gamma $}} by 
$$F(t,\gamma )=\log \vert \varphi _{t}(\gamma )\vert 
_{t}.$$
 By 14.7 of 
\cite{R1} (and I am sure this is well-known), this function has a remarkable property. There is $t(\gamma )\in 
\mathbb R$, and  a constant $C_{0}$ depending only on  the 
topological type of $S$ and a bound on the number of loops in 
$\gamma $,  such that
\begin{equation}\label{7.20.1}
    \vert F(t,\gamma )-F(t(\gamma ),\gamma )-\vert t-t(\gamma )\vert \vert \leq 
    C_{0}.\end{equation}
 As a consequence of this, and of the theory of section \ref{5}, we have the following. As already noted in \ref{7.0}, this will not be sufficient for our purposes, but it is a start, and is, in fact, used in the construction of $\alpha _1$.

\begin{lemma}\label{7.6} Fix a topological surface $S$ and ltd parameter functions $(\Delta ,r,s,K_0)$ and $\nu _0>0$. The following holds for  $C_1>0$ depending only on the topological type of $S$, and constants $s_0$ and $L$ depending on the topological type of  $S$ and $(\Delta ,r,s,K_0)$ and $\nu _0$, where $L$ is as in \ref{5.5}, and for  any sufficiently large $\Delta _1>0$ given these.  Let $[y_0,y_T]=[[\varphi _0],[\varphi _T]]=\{ [\varphi _t]:t\in [0,T]\} $ be a Teichm\"uller geodesic segment, and  $\ell \subset [y_0,y_T]$ a segment of length $\ge \Delta _1$. Let $\beta $ be disjoint from all subsurfaces $\alpha $ of $S$  such that $(\alpha ,\ell ')$ is ltd with respect to $(\Delta ,r,s,K_0,\nu _0)$ for some $\ell '\subset \ell $. Then for all $t\in[0,T]$,
\begin{equation}\label{7.13.3}s_0a'(\beta )\le |\varphi _t(\partial \beta )|_t^2,\end{equation}
and hence
 \begin{equation}\label{7.13.4}a'(\beta )\leq C_{1}s_0^{-1}L^{2}e^{-\Delta 
_{1}}.\end{equation}
\end{lemma}
\begin{proof} Enlarge $\beta $ if necessary, so that $\partial \beta $ is contained in the convex hull of the union of $\partial \alpha $ such that $(\alpha ,\ell ')$ is ltd and $\ell '\subset \ell $. By \ref{5.5}, for $[\varphi ]\in \ell $:
\begin{equation}\label{7.13.2}\vert \varphi (\partial \beta )\vert 
\leq 
L.\end{equation}
 So if $\ell \subset [y_0,y_T]$, the function $F(t,\partial \beta )=|\varphi _t(\partial \beta )|_t$ is bounded, in terms of $L$, on an interval of length $\ge \Delta _1$.  There is also a constant $C_0$ depending only on the topological type of $S$ such that 
 $$|\varphi (\gamma )| _q\le C_0|\varphi (\gamma ) |$$
 for any closed loop $\gamma $ on $S$ and $[\varphi ]\in {\cal{T}}(S)$ and quadratic differential $q(z)dz^2$ at $[\varphi ]$. We apply this with $\varphi =\varphi _t$ and $q=q_t$ for varying $t$. So by (\ref{7.20.1}), there is at least one $t$ such that
$$|\varphi _t(\partial \beta )|_t\le C_1e^{-\Delta _1/2}L$$
Let $s_0=s_0(\Delta ,r,s,K_0)>0$ as in \ref{5.4}. By the last part of \ref{5.4}, we have
\begin{equation}\label{7.13.1}e ^{2F(t,\partial \beta )}=\vert \varphi 
_{t}(\partial \beta 
)\vert _{t}^{2}\geq  s_{0}a'(\beta ),\end{equation} 
which gives (\ref{7.13.3}).  But $a'(\beta )$ is independent of $t$. So (\ref{7.13.4}) follows, with $C_1=C_0^2$.
\end{proof} 

\subsection{Comparision between Poincar\'e length and qd-length}\label{7.7}

Comparision between Poincar\'e and qd-length can  be made as 
follows. Given $L_{1}>0$ there is $L_{2}\in \mathbb 
R$ such that if
\begin{equation}\label{7.20.2}
    \vert \varphi _{t}(\gamma )\vert \leq 
L_{1}\end{equation}
    then 
    \begin{equation}\label{7.20.3} F(t,\gamma )-F(t,\gamma ')\leq 
L_{2}\end{equation}
for all nontrivial 
nonperipheral $\gamma '$ intersecting $\gamma $ transversely. 
Conversely, given $L_{2}\in \mathbb R$, there is $L_{1}$ such that 
(\ref{7.20.2}) holds whenever (\ref{7.20.3}) holds for all $\gamma '$ 
intersecting $\gamma $ transversely. There is a similar 
characterisation of short loops. Given $L_{2}<0$, there is $L_{1}>0$ 
(which is small if $L_{2}$ is negatively large)
such that, whenever (\ref{7.20.2}) holds, then (\ref{7.20.3}) holds  
for all $\gamma '$ intersecting 
$\gamma $ tranversely. Conversely, given $L_{1}>0$, there is $L_{2}$ 
(which is negative if $L_{1}$ is small) such that (\ref{7.20.2}) holds for 
$\gamma $, whenever (\ref{7.20.3}) holds for $\gamma $ and all 
$\gamma '$ transverse to $\gamma $.

 We now start to deal with the problem we identified in \ref{7.0}:  given $t$, how to bound the intersection of $\varphi _{t}(\beta )$ with a  segment $\zeta _1$ of unstable foliation of bounded Poincar\'e length, for varying $t$. In the following series of lemmas \ref{7.22} - \ref{7.15}, we split up the segments of unstable foliation across $\varphi _t(\beta )$ into sets which are dealt with separately, some of them encased in larger subsurfaces. We then have to make estimates on these larger subsurfaces, and also on the number of them.
 
 \begin{lemma}\label{7.22} Let $\gamma $ be an arc such that $\varphi _t(\gamma )$ is a segment of unstable foliation for one, and hence all, $t$. Given $L$ there is $L_1$ such that if $\varphi _t(\gamma )$ has Poincar\'e length $\le L$ then $\varphi _u(\gamma )$ has Poincar\'e length $\le L_1$ for all $u\le t$. Moreover, if $\varepsilon _0$ is any Margulis constant  then given $C>0$ there exists $C_1>0$ such that, if  $\gamma \subset (\varphi _t(S))_{<\varepsilon _0}$ or $\gamma \subset (\varphi _t(S))_{\ge \varepsilon _0}$, and $\varphi _t(\gamma )$ has Poincar\'e length $\le C$ times the injectivity radius, then $\varphi _u(\gamma )$ has Poincar\'e length $\le C_1$ times the injectivity radius, for all $u\le t$. 
 
 Similar statements  hold for stable segments and $u\ge t$.\end{lemma}
 
 \begin{proof} Since a segment of length $\le L$ can be split up into a number of segments of length bounded by the injectivity radius, where the number is bounded in terms of $L$, the first statement follows from the second. So take any such segment $\gamma $ of length bounded by the injectivity radius at a point on $\gamma $, and take any closed loop $\zeta $ intersecting $\gamma $. Then $\log |\varphi _u(\gamma )|_u-\log |\varphi _u(\zeta )|_u$ is non-increasing for $u\le t$. The result follows.\end{proof}
 
 \begin{lemma}\label{7.19} Given $\eta _1>0$ and $L>0$ there exists $\eta _2>0$ depending only on $\eta _1$, $L$ and the topological type of $S$, such that the following holds for any sufficiently large $\Delta _1>0$. Let $\beta $ be a gap or loop with $|\varphi _t(\partial \beta )|\le L$ and $a'(\beta )\le e^{-\Delta _1}$. Then there is $\beta '\supset \beta $ such that $a'(\beta ')\le e^{-(1-\eta _1)\Delta _1}$ and $|\varphi _u(\partial \beta ')|\le e^{-\eta _2\Delta _1}$ for $|u-t|\le \eta _2\Delta _1$.\end{lemma}
 \begin{proof} Choose $\varepsilon _0>0$ such that no loop of Poincar\'e length $\le \varepsilon _0$ is intersected transversely by a loop of length $\le L$. For any $\eta _3>0$ sufficiently small depending only on the topological type of $S$ we can find a connected union $\beta '$ of gaps and loops $\omega $ containing $\beta $ such that $|\varphi _t(\partial \omega )|\le \varepsilon _0$, and if $\omega '$ is adjacent to $\omega $ and also in this union then  $e^{-\eta _3\Delta _1}\le a'(\omega ')/a'(\omega )\le e^{\eta _3\Delta _1}$, but if $\omega $ is inside the union and $\omega '$ outside then $a'(\omega ')>e^{\eta _3\Delta _1}a'(\omega )$. It follows that for any such $\omega $ and $\omega '$ we have
 $$|\varphi _t(\partial \omega '\cap \partial \omega )|\le e^{-\eta _3\Delta _1/2}{\rm{Max}}(a'(\omega ),a'(\omega ')).$$
 So then for $|u-t|\le \eta _3\Delta _1/4$ we have
  $$|\varphi _u(\partial \omega '\cap \partial \omega )|\le e^{-\eta _3\Delta _1/4}{\rm{Max}}(a'(\omega ),a'(\omega ')).$$
  Now there are at most $p$ sets $\omega $ in the union, where $p$ depends only on the topological type of $S$ and hence
  $$a'(\beta ')\le e^{p\eta _3\Delta _1}a'(\beta )\le e^{-(1-p\eta _3)\Delta _1}.$$
  So if we choose $\eta _3$ with $p\eta _3\le \eta _1$ we have $a'(\beta ')\le e^{-(1-\eta _1)\Delta _1}$ and the bound on $|\varphi _u(\partial \beta ')|$ holds for $\eta _2=\eta _3/4$.\end{proof}

\begin{lemma}\label{7.11}   Let $|\varphi _{u_1}(\partial  \beta)|\le L$ and $|\varphi _{u_1}(\partial \beta )|_{u_1,-}\ge \frac{1}{2}|\varphi _{u_1}(\partial \beta )|_{u_1}$,
and let $a'(\beta )\le e^{-\Delta _1}$. Then if $\Delta _1$ is sufficiently large given constants $\mu _i$, $L$ and the topological type of $S$,  there is an increasing sequence of subsurfaces $(\beta _i:1\le i\le k)$ with $\beta =\beta _1$,  sequences of surfaces $(\omega _i=\omega _i(\beta ):1\le i\le k)$ and $(\omega _i'=\omega _i'(\beta ):1\le i\le k-1)$ with $\omega _1=\beta $,   $\omega _k=\emptyset $, a decreasing sequence of real numbers $u_i=u_i(\beta )$, and constants $L_i$ depending only on $L$ and $\mu _j$ for $j<i$ and the topological type of $S$, such that the following hold.
\begin{itemize}
\item  $\beta _i\setminus \beta _{i-1}\subset \omega _i\subset \beta _i$  and $\omega _i'\subset \omega _i\cap \omega _{i+1}$, so that $\beta _k=\cup _{i=1}^{k-1}\omega _i\setminus \omega _{i+1}$. 
\item $|\varphi _{t}(\partial \omega _i)|\le L_i$ for $u_i\le t\le u_{i-1}$, for $2\le i\le k$, 
 $|\varphi _{u_i}(\gamma )|_{u_i,-}\ge \frac{1}{2}|\varphi _{u_i}(\gamma )|_{u_i}$ for some component $\gamma $ of $\partial \beta _i$ with $|\varphi _{u_i}(\gamma )|\ge 1$.
\item  For all $t\le u_i$, every  segment $\zeta $ of unstable foliation across $\varphi _{t}(\omega _{i}\setminus \omega _i')$ has Poincar\'e length $\le L_i$ for all $t\le u_i$, and every maximal segment  $\zeta $ of unstable foliation across $\varphi _t(\omega_i\setminus \omega _{i+1})$ is adjacent on each side to a segment $\zeta '$ of unstable foliation in $\varphi _t(S\setminus \omega  _{i})$ with $|\zeta |_t\le \mu _i|\zeta '|_t$.
\item $a'(\beta _{i+1})\le CL_{i}\mu _{i}^{-1}a'(\beta _{i})$, where $C$ depends only on the topological type of $S$. 
\item Either $\beta _{i+1}$ is strictly bigger than $\beta _i$ or $\omega _{i+1}$ is strictly smaller than $\omega _i$.
\item Either $u_k=0$ or  every unstable segment across $\varphi _{u_k}(\beta _k)$ has Poincar\'e length $\le L_k$ and every unstable segment $\zeta $ across $\varphi _{u_k}(\beta _k)$ is adjacent to a segment $\zeta '$ of $\varphi _{u_k}(S\setminus \beta _k)$ with $|\zeta |\le \mu _k|\zeta '|$. 
\end{itemize}
\end{lemma}
\noindent {\textbf{Remark}} Note that there is no claim that $|\varphi _{u_i}(\beta _{i-1})|$ is bounded.

\begin{proof} We start the inductive definitions with $\beta =\beta _1=\omega _1$. By the Fundamental Lemma \ref{7.1}, for a constant $L_1$ depending only on $L$ and the topological type of $S$, there is a subsurface $\omega _1'$ of $\beta $ --- which could be empty --- such that $|\varphi _{u_1}(\partial \omega _1')|\le L_1$ and 
$$|\varphi  _{u_1}(\partial \omega _1')|_{u_1,-}\le L^{-1}|\varphi _{u_1}(\partial \omega _1')|_{u_1}.$$
and every unstable segment across $\varphi _{u_1}(\beta \setminus \omega _1')$ has Poincar\'e length $\le L_1$. The same estimate for $u\le u_1$ follows by \ref{7.22}, for $L_1$ sufficiently large given $L$.

  Let $\beta _2'$ be the surface which is the union of $\beta =\beta _1$ and every maximal unstable segment $\zeta '$ outside $\beta $ which is adjacent to an unstable segment $\zeta $ across $\varphi _{u_1}(\beta \setminus \omega _1')$, such that $|\zeta '|_{u_1}<\mu ^{-1}|\zeta |_{u_1}$. Then $a'(\beta _2')\le (1+\mu ^{-1})a'(\beta )$. Now let $\beta _2$ be the surface containing $\beta _2'$ such that $(\beta _2\setminus \beta _2')\cup (\beta _2'\setminus \beta _2)$ is a union of discs and annuli, and such that $\varphi _{u_1}(\partial \beta _2)$ is in good position. Let $\varphi _{u_1}(\omega _2)\subset \varphi _{u_1}(\beta _2)$ be the union of $\varphi _{u_1}(\omega _1')$ and the surface obtained by leaving out those unstable segments $\zeta $ in $\varphi _{u_1}(\beta )$ for which an adjacent unstable segment $\zeta '\subset \varphi _{u_1}(S\setminus \beta )$ exists on each side, with $|\zeta |_{u_1}\le \mu |\zeta '|_{u_1}$, again homotoping so that $\varphi _{u_1}(\omega _2)$ is in good position. Thus $\varphi _{u_1}(\omega _2)$ is obtained from $\varphi _{u_1}(\beta _2)$ by leaving out some handles with boundary which was in $\varphi _{u_1}(\partial \beta )$, before the good position homotopy. Then $|\varphi _{u_1}(\omega _2)|\le \mu ^{-1}L_1$, assuming that $CL\le L_1$, for  a suitable constant $C$ depending only on the topological type of $S$. 
  
  Now we prove that $a'(\beta _2)\le C\mu ^{-1}a'(\beta )$, again assuming that $C$ is suitably chosen. The surface $\varphi _{u_1}(\beta _2\setminus \beta _2')$ is a union of topological discs and annuli. So the aim is to bound the areas of the added discs and annuli. Each topological disc  is a polygon with each side at  a constant slope to the stable and unstable foliation, with alternate sides tangent to the unstable foliation. The number of sides of each polygon is bounded by the number and type of singularities inside the polygon, which is bounded by the topological type of $S$. So the total number of polygons in $\varphi _{u_1}(\beta _2\setminus \beta _2')$ which have more than four sides is bounded by the topological type of $S$. 
  
  The area of any polygon with at most four sides -- a trapezium -- is bounded by the product of the length of two adjacent sides, that is, by $\mu ^{-1}$ times the product of the lengths of an adjacent triangle or trapezium in $\varphi _{u_1}(\beta )$. For any of the other boundedly finitely many polygons, we foliate by unstable segments, and so obtain the polygon as a finite union of trapezia, such that only unstable sides of trapezia can be in the interior of the polygon, and sides which are not unstable segments are in  $\varphi _{u_1}(\partial \beta )$. So then, by induction, we obtain a bound on the area of the polygon in terms of $\mu ^{-1}$ times the area of adjacent trapezia in $\varphi _{u_1}(\beta )$. So the area of the polygon is $\le C_1\mu ^{-1}a'(\beta )$ for a suitable constant $C_1$.  
  
  To obtain a similar area bound for the annuli in $\varphi _{u_1}(\beta _2\setminus \beta _2')$, it suffices to bound the number of polygons in an annulus $\varphi _{u_1}(A)$ bounded by $\varphi _{u_1}(\gamma )$ and $\varphi _{u_1}(\gamma ')$, where $\gamma $ and $\gamma '$ are homotopic components of $\partial \beta _2$ and $\partial \beta _2'$ respectively. Since $\varphi _{u_1}(\gamma )$ is in good position, we have a bound, in terms of the topological type of $S$, on the number of constant slope segments on $\varphi _{u_1}(\gamma )$. We need a bound on the number of constant slope segments on $\varphi _{u_1}(\gamma '$, which means bounding the number of segments of $\varphi _{u_1}(\partial \beta )$ and $\varphi _{u_1}(\partial \beta _2')$ on $\varphi _{u_1}(\gamma ')$. To do this, we consider the trapezia in $\varphi _{u_1}(\beta _2'\setminus \beta )$ which are subsets of the boundedly finitely many trapezia in $\varphi _{u_1}(S\setminus \beta )$. Both sets of trapezia are foliated by unstable segments and have their other sides in $\varphi _{u_1}(\partial \beta )$. No two trapezia of $\varphi _{u_1}(\beta _2'\setminus \beta )$ can be in the same trapezium of $\varphi _{u_1}(S\setminus \beta )$ and bounded by the same trapezia of $\varphi _{u_1}(\beta )$. So the number of trapezia in $\varphi _{u_1}(\beta _2'\setminus \beta )$ is bounded. So the number of boundary components of these trapezia is bounded, and hence the number that can intersect $\varphi _{u_1}(\gamma ')$ is bounded.

 If $\beta _2=\beta _1$  and $\omega _2=\omega _1'= \emptyset $, then we define $k=1$.  Now suppose that $\beta _2\ne \beta _1$ and $\omega _2\ne \emptyset $.  Then choose the first $u_2$ to be the first $t\le u_1$ such that $|\varphi _{t}(\gamma )|_{t,-}=\frac{1}{2}|\varphi _t(\gamma )|_t$ for a component $\gamma $ of $\partial \beta _2$ with $|\varphi _t(\gamma )|\ge 1$.  Then, from the bound  on $|\varphi _{u_1}(\partial \omega _1'\cup \partial \beta )|$, we have $|\varphi _t(\partial \omega _2)|\le L_2$ for $u_2\le t\le u_1$, for $L_2$  depending only  on $L$ and $\mu _1$ . By the Fundamental Lemma \ref{7.1}, enlarging $L_2$ if necessary, but still depending only on $L$ and $\mu _1$, there is $\omega _2'\subset \omega  _2$ (where $\omega _2'$ is allowed to be empty) such that $|\varphi _{u_2}(\partial \omega _2')|\le L_2$, and $|\zeta |\le L_2$ for every unstable segment $\zeta $ across $\varphi _{u_2}(\omega _2\setminus \omega _2')$  has Poincar\'e length $\le L_2$ and 
 $$|\varphi _{u_2}(\partial \omega _2')|_{u_2,-}\le L^{-1}|\varphi _{u_2}(\partial \omega _2)|_{u_2}$$
 We define $\varphi _{u_2}(\beta _3')$ to be the union of $\varphi _{u_2}(\beta _2)$ and of all unstable  segments $\zeta '$ in $\varphi _{u_2}(S\setminus \omega  _2)$ such that $\zeta '$ has both  endpoints in $\varphi _{u_2}(\partial \omega _2)$ and $\zeta '$ is adjacent to an unstable segment $\zeta $ in $\varphi _{u_2}(\omega _2\setminus \omega _2')$ with $|\zeta '|\le \mu _2^{-1}|\zeta |$. We therefore have $a'(\beta _3')\le \mu _2^{-1}a'(\beta _2)$. We then define $\beta _3$,  $\omega _3$  and $L_3$ from $\beta _3'$, $\omega _2'$, $L_2$ and $\mu _2$ in exactly the same way as $\beta _2$, $\omega _2$ and $L_2$ are defined from $\beta _2'$, $L$ and $\mu $, and continue to define $u_3$ and $\omega _3'$ analogously to $u_2$ and $\omega _2'$. The definition of $\beta _i$, $\omega _i$, $\omega _i'$, $L_i$ and $u_i$ is exactly the same for all $i\ge 3$. The bound on $a'(\beta _i)$ for $i\ge 3$ works in the same way as the bound on $a'(\beta _2)$.

Since any sequence of subsurfaces of $S$  of strictly increasing or strictly decreasing topological type is bounded --  in terms of the topological type of $S$, there is $k$ bounded in terms of the topological type of $S$ such that $\beta _k=\beta _{k-1}$ and $\omega _k=\emptyset $ with $\beta \subset \beta _k$ and $a'(\beta _k)\le \prod _{i\le k}L_i\mu _i^{-1}a'(\beta )$. If $\Delta _1$ is sufficiently large, it follows that $\beta _k\ne S$.\end{proof}

For the sets $\beta _j=\beta _j(\beta )$ and $\omega _j=\omega _j(\beta )$ as in \ref{7.11}, we have 
$$\beta _{i-1}\setminus \omega _i=\cup _{j=1}^{i-1}(\omega _{j+1}\setminus \omega _j).$$

\begin{lemma}\label{7.15} Fix a Teichmuller geodesic segment  $\ell _0=[y_{s_1},y_{s_0}]$ with $|\ell _0|\le   p_1\Delta _1$. The number of $\beta $ such that $|\varphi _t(\partial \beta )|\le L$ for $t$ in an interval of $[s_1,s_0]$ of length $\ge\Delta _1/p_1$ is bounded in terms of $L$,  $p_1$  and the topological type of $S$. Let $u_i(\beta )$ and $\omega _i(\beta )$ be as in \ref{7.11}. The number of $\omega _i=\omega _i(\beta )$ with $y_{u_i}\in \ell _0$ is bounded in terms of $L$,  $p_1$, $(\mu _j:j<i)$  and the topological type of $S$,  in any interval of length $\Delta _1$, for any $i\le k=k(\beta )$, even if $y_{u_1}\notin \ell _0$.\end{lemma}
\begin{proof} The loop set $\partial \beta $ is a {\em{multicurve}} (see \ref{2.0}), that is, a set of simple closed loops which are homotopically disjoint and distinct. For any fixed $t$, the number of multicurves  in $\varphi _t(S)$ of length $\le L$ is $\le C_1L^{6g-6+2b}$ where $g$ is the genus of $S$ and $b$ the number of boundary components and $C_1$ is a universal constant -- just depending on the Margulis constant in two dimensions. So by considering $p_1^2+1$ points in $\ell _0$ such that any other point of $\ell _0$ is distance $\le \Delta _1$ from one of these, we see that the number of choices for $\beta $ is $\le L^{(p_1^2+1)(6g-6+2b)}$ if $y_{t_1}\in \ell $. Now suppose that $y_{u_i}\in \ell _0$. If $y_{u_1}\notin \ell _0$ then $s_0\in [u_{j},u_{j-1}]$ for some $j\le i$. Since $|\varphi _t(\partial \omega _j)|\le L_j$ for $t\in [u_{j},u_{j-1}]$, this is true for $t=s_0$, and hence the number of choices for this  $\omega _j$ is  $\le C_1L_j^{6g-6+2b}$, where $L_j$ depends only on $L$ and $\mu _{j'}$ for $j'<j$. Then $u_j$ is determined from $\omega _j$ to within a bounded distance by the property $|\varphi _{u_j}(\gamma )|_{u_j,-}\ge \frac{1}{2}\vert \varphi _{u_j}(\gamma )|_{u_j}$ for a component $\gamma $ of $\partial \omega _j$ with $|\varphi _{u_j}(\gamma )|\ge 1$. We also have $|\varphi _{u_j}(\partial \omega _{j+1})|\le L_j$. So the number of choices for $\omega _{j+1}$, given $u_j$, is also $\le C_1L_j^{6g-6+2b}$. Then from  $\omega _{j+1}$ and the predetermined $L_{j+1}$ we can determine $u_{j+1}$, and hence we have a bound on the number of choices for $\omega _{i'}$ for all $i'\le i$ which depends only on $p_1$, $\mu _{i'}$ for $i'<i$ and the topological type of $S$.

\end{proof}

We now have the estimates in place to bound the intersection of unstable segments with ``bad set'', that is, the set of $\beta $ bounded by $L$. But there is still some work to do on the complement, the ``good set'' because this is the convex hull of long thick and dominants, rather than their union. So we need the following.

\begin{lemma}\label{7.17} Fix ltd parameter functions $(\Delta ,r,s,K_0)$ and $\nu _0$ as in \ref{5.4}.There is a constant $M$ depending on these such that the following holds. Let $\ell =[y_{u_1},y_{u_0}]\subset [y_0,y_t]$ be any Teichm\"uller geodesic segment such that there is at least one ltd $(\alpha ,\ell ')$ with $\ell'\subset \ell $. Let $\Omega $ be the convex hull of all such $\alpha $ and let $\Omega '$ be the union of all such $\alpha $. Then   any unstable segment $\zeta $ of $\varphi _t(\Omega \setminus \Omega ')$ is adjacent to a unstable segment $\zeta '$ in $\varphi _t(\alpha )$ for some ltd $(\alpha ,\ell ')$ with $\ell '\subset \ell $   with $|\zeta '|_t\ge |\zeta |_t/M$ for all $t$. A similar statement holds for stable segments.\end{lemma}

\begin{proof}  We treat the case  of unstable segments. The proof for  stable segments is exactly analogous. There are finitely may ltd's $(\alpha _i,[y_{w_i},y_{v_i}])$ with $1\le i\le k$ and $w_{i}\le w_{i+1}$ and $u_1\le w_i$ for all $i$, such that $\Omega _{i}$ is of larger topological type than $\Omega _{i-1}$, where $\Omega _i$ is the convex hull of $\cup _{j\le i}\alpha _j$ and $\Omega _0=\emptyset $, and there is no ltd $(\alpha ,\ell )$ such that $\alpha $ has essential intersection with $\Omega _{i}\setminus \Omega _{i-1}$ and $\ell \subset [y_{u_1},y_{w_i}]$.  We say that $\alpha _i$ is {\em{visible  from $y_{u_1}$}}, meaning that parts of it are. The visibility property means that, by \ref{5.5}, there is a constant $L_1$ depending only on $(\Delta ,r,s,K_0,\nu _0)$ and the topological type of $S$  such that  $|\varphi _{u}(\partial \Omega _{i})|\le L_1$ for $u_1\le u\le {\rm{\max}}(v_i,w_{i+1})$,  for each $i$.

We then aim to show inductively that, for a constant $c>0$ depending only on $(\Delta ,r,s,K_0,\nu _0)$ and the topological type of $S$, and any unstable segment $\zeta $ across $\varphi _u(\Omega _i)$,
\begin{equation}\label{7.17.1}|\zeta |_u\ge c|\zeta \cap (\cup _{j\le i}\varphi _u(\alpha _j))|_u\end{equation}
 for all $u$, and all $i$, and, if $\gamma $ is a component of $\partial \Omega _i$ such that $\varphi _u(\gamma )$ contains  an endpoint of $\zeta $, then, if $u\ge w_i$,
 \begin{equation}\label{7.17.2}|\varphi _u(\gamma )|_u\ge c|\zeta |_u.\end{equation}
These statements suffice, because $\gamma $ is a union of boundedly finitely many segments in the sets $\alpha _j$ for $j\le i$, where, for each $j$,  each segment of $\gamma \cap \alpha _j$ is disjoint from all ltd's $(\alpha ',\ell ')$ with $\ell '\subset [y_{u_1},y_{w_j}]$. In both cases, if the inequalities hold for $u=w_i$, they hold for all claimed $u$.  For (\ref{7.17.1}), this is because for different $u$, the two sides of the inequality are scaled by the same factor. For (\ref{7.17.2}), the lefthand side of the inequality is obtained from that for $u=w_i$ by multiplying by $e^{u-w_i}$ and the right-hand side is obtained by mutliplying by at most $e^{u-w_i}$. The statements are trivially true for $\Omega _1=\alpha _1$, by considering $u=w_1$. So we consider the inductive statements. We assume they are true for $\Omega _{i-1}$ with $i\ge 2$. Now  $\varphi _{u}(\Omega _i)$ is obtained from $\varphi _{u}(\Omega _{i-1}\cup \alpha _i)$ by adding annuli and topological discs, each of which is bounded by transversally intersecting components of $\varphi _u(\partial \Omega _{i-1})$ and $\varphi _u(\partial \alpha _i)$. For $u=w_i$, if $\gamma $ and $\gamma '$ are transversally intersecting components of $\partial \Omega _{i-1}$ and $\partial \alpha _i$, then  $|\varphi _u(\gamma )|$ and $|\varphi _u(\gamma ')|$ (that is, the Poincar\'e lengths) are boundedly proportional, with bound depending on the ltd parameter functions.  Therefore $|\varphi _{w_i}(\gamma )|_{w_i}$ and $|\varphi _{w_i}(\gamma ')|_{w_i}$ are also boundedly proportional.  So (\ref{7.17.2}) is true by induction --- possibly after modfifying $c$, but since this only has to be done for each $i$ and the number of $i$ is bounded in terms of the topological type of $S$, this is allowed. Unstable segments across any added topological disc or annulus have $qd$-length  bounded by a constant times $|\varphi _{w_i}(\gamma )|_{w_i}$ for any component $\gamma $ of $\partial \Omega _{i-1}$ or $\partial \alpha _i$ such that $\varphi _{w_i}(\gamma )$ intersects the boundary of this disc or annulus. This is less than $C_1|\zeta |_{w_i}$ for any  adjacent  unstable segment  $\zeta $ in $\varphi _{w_i}(\Omega _{i-1})$, by (\ref{7.17.2}) for $i-1$ replacing $i$,  and for  any unstable segment $\zeta $ in $\varphi _{w_i}(\alpha _i)$, since $(\alpha _i[y_{w_i},y_{v_i}])$ is ltd.  So we obtain
 $$|\zeta |_{w_i}\ge c_1|\zeta \cap \varphi _{w_i}(\Omega _{i-1}\cup \alpha _i)|_{w_i},$$
 for a suitable $C_1>0$, and then (\ref{7.17.1}) also follows, by induction.

\end{proof} 

\subsection{Proof of \ref{7.13}: construction of the 
sequences.}\label{7.18}
For $1\leq i\leq R_{0}$, some $R_{0}$, we shall 
find   sequences  $t_{i}$,  $\alpha _{i}$, $\zeta _{i}$, such that the 
following hold. We start by choosing $t_1$ with $0\le t_1\le \Delta _0/2$. But after that, to simplify the writing, we assume that $t_1=0$. In the most technical property, 5, the constant $\mu $ is as in \ref{7.11}, and $p$ is any integer such that $1/p\le \eta _2/3$, for $\eta _2$ as in \ref{7.19}. For any suitable $\mu $ and $p$, property 5 will be obtained if $\Delta _0$ is large enough. 
\begin{enumerate}
\item[1.] $(\alpha _{i},\ell _i)$ is ltd at $y_{t_{i}}\in \ell _i\subset  [y_{0},y_{T}]$ with respect to $(\Delta ,r,s,K_0)$, in the first quarter of $\ell _i$ if $\alpha _i$ is a loop, and $a'(\alpha _i)\ge e^{-\Delta _0/2}$.
\item[2.] $\zeta _{i+1}\subset  \zeta _{i}$ and $\zeta _i\subset \varphi _{0}(\alpha _{i})$ if $\alpha _i$ is a gap, and $\zeta _i\subset \varphi _0(A(\alpha _i))$ if $\alpha _i$ is a loop, where $\varphi _t(A(\alpha _i))$ is the flat annulus in the $q_t$-metric which is homotopic to $\varphi _t(\alpha )$.
\item[3.] Writing $t_1=0$, for all $i\geq 1$, 
$\chi _{t_{i}}(\zeta _{i})$ is a
segment of the unstable foliation of the quadratic differential $q_{t_i}(z)dz^2$ whose Poincar\'e length is boundedly proportional to the injectivity radius at that point of 
$\varphi _{t_i}(S)$. 
\item[4.] $t_{1}\leq \Delta _{0}/2$ and $t_{R_{0}}\geq T-\Delta _{0}$. For all $1\le i<R_{0}$, $t_{i}<t_{i+1}\leq 
t_{i}+\Delta _{0}$.
\item[5.]  Each segment $\varphi _{t_i}(\gamma )$ of  $\chi _{t_{i}}(\zeta _{i})\cap \varphi _{t_i}(\omega_j(\beta )\setminus \omega _{j+1}(\beta ) )$ is adjacent to a segment of  $\chi _{t_{i}}(\zeta _{i})\setminus  (\varphi _{t_i}(\omega_j(\beta )\setminus \omega _{j+1}(\beta )))$ whose (Poincar\'e or qd) length  is at least $\mu _j^{-1/2}$ times more, for any $\beta $  with $a'(\beta )\le e^{-\Delta _0/3}$ and $|\varphi _t(\partial \beta )|\le L$ for  $|t-u_0|\le \Delta _0/p$  and $u_0=u_0(\beta )\ge t_i+\Delta _0$. If $|\varphi _{t_i}(\gamma )|_{t_i}\ge \mu ^{-1}|\chi _{t_i}(\zeta _i)|_{t_i}$ for such a $\gamma $, then $\varphi _{t}(\gamma )\cap \chi _{t}(\zeta _{i+1})=\emptyset $ for all $t$.\end{enumerate}

1,2 and 4 give the proof of \ref{7.13}, apart from $a'(\alpha _i)>\delta _0$,  since, by 2, we have 
$$\zeta _{j}\subset  \varphi _{t_1}(\alpha _{i}) $$ 
for all $i\leq j$. 3 and 5 are needed for the inductive process. The notation of Property 5 comes from \ref{7.11}. 

Note that property 5 implies that $a'(\alpha _i)>e^{-5\Delta _0/14}$, which gives $a'(\alpha _i)>\delta _0$, if $\delta _0=e^{-5\Delta _0/14}$. For if $a'(\alpha _i)\le e^{-5\Delta _0/14}$, then by \ref{7.19} there is $\beta \supset \alpha $ with $|\varphi_t(\partial \beta )|\le e^{-\Delta _0/p}$ for $|t-t_i|\le \Delta _0/p$ and $a'(\beta )<e^{-\Delta _0/3}$,  for a suitably chosen $p$ depending only on the topological type of $S$. But then by Property 5 for this $\beta $, and the properties of unstable segment across $\varphi _t(\omega _j\setminus \omega _{j+1})$ of \ref{7.11},  $\zeta _i$ is not contained in $\cup _j\varphi _0(\omega _j(\beta )\setminus \omega _{j+1}(\beta ))$, assuming the $\mu _j$ grow sufficiently fast. But since $\beta \subset \cup _j\omega _j(\beta )\setminus \omega _{j+1}(\beta )$, this gives the contradiction that $\alpha $ is not contained in $\beta $. So we do have $a'(\alpha _i)>e^{-5\Delta _0/14}$. Then since $\chi _{t_i}(\zeta _i)$ has Poincar\'e length bounded from $0$ if $\alpha _i$ is a gap and boundedly proportional to the injectivity radius if $\alpha _i$ is a loop, we obtain $|\chi _{t_i}(\zeta _i)|_{t_i}\ge e^{-5\Delta _0/29}$ in both cases, provided that, when $\alpha _i$ is a loop,  we choose $t_i$  such that $y_{t_i}$ is towards the right end of the segment $\ell _i$ along which $(\alpha _i,\ell _i)$ is ltd, which we can do by inserting extra points $t_j$ with $j<i$  $y_{t_j}\in \ell _i=\ell _j$ and $\alpha _j=\alpha _i$, if necessary. Now if $t$ is chosen so that $t_i+\Delta _0/2\le t\le t_i+\Delta _0$, and $\zeta \subset \zeta _i$ is any segment such that $|\chi _t(\zeta )|$ is bounded, then $|\chi _t(\zeta )|_t$ is also bounded, and $|\chi _{t_i}(\zeta )|_{t_i}\le e^{-\Delta _0/2}|\chi _t(\zeta )|_t$. So 
\begin{equation}\label{7.13.1}|\chi _{t_i}(\zeta )|_{t_i}/|\chi _{t_i}(\zeta _i)|_{t_i}\le e^{-\Delta _0/4}.\end{equation}
In particular, (\ref{7.13.1}) holds if $t_i+\Delta _0/2\le t\le t_i+\Delta _0$ and $\chi _t(\zeta )$ is an unstable segment of bounded length in $\varphi _t(\alpha )$, where $(\alpha ,\ell )$ is ltd with $[\varphi _t]\in \ell $. (\ref{7.13.1}) also holds if $t=t_i+\Delta _0/2$, and $\chi _t(\zeta )$ is an unstable segment across $\varphi _t(\Omega \setminus \Omega ')$, where $\Omega '$ is the union of all $\alpha $ such that $(\alpha ,\ell )$ is ltd for some $\ell \subset [y_{t_i+\Delta _0/2}y_{t_i+\Delta _0}]$, and $\Omega $ is the convex hull of $\Omega '$. This is because, if $\sigma $ is any component of $\Omega \setminus \Omega '$ --- and hence $\sigma $ is a disc or annulus -- then $|\varphi _t(\partial \sigma )|$ is bounded.   (\ref{7.13.1}) is also true if $\chi _t(\zeta )$ is the union of a (possibly empty) segment across $\varphi _t(\Omega \setminus \Omega ')$ and an adjacent segment $\chi _t(\zeta ')$in $\varphi _t(\alpha )$ for an $\alpha $ such that $(\alpha ,\ell )$ is ltd with $\ell \subset [y_{t_i+\Delta _0/2},y_{t_i+\Delta _0}]$,   with $\chi _t(\zeta ')|_t\ge |\chi _t(\zeta )_t/M$, and such that $\chi_t(\zeta ')$ has Poincar\'e length at least  multiple bounded form $0$ injectivity radius at any point of $\chi _t(\zeta ')$, for any $[\varphi _t]\in \ell $.

 In what follows, we write $\chi _t(\zeta _i)\cap \varphi _t(\Omega )$ as 
 $$\cup _{\zeta \in A}\chi _{t}(\zeta )$$
  for one, and hence any $t$, for a set $A$ of such segments, that is, including a segment of some $\varphi _0(\alpha )$, where any two distinct segments in $A$ have disjoint interiors.

The argument for finding $(\alpha _1,\ell _1)$ is different from the argument for $(\alpha _i,\ell _i)$ for $i>1$. We take our fixed ltd parameters $(\Delta ,r,s,K_0)$. Let $\nu _0$ be as given by \ref{5.4} and $L$ be given by \ref{5.5} for $(\Delta ,r,s,K_0)$ and $\nu _0$. Let $L_0$ be such that $L_j\le L_0$ for all $L_j$ arising as in \ref{7.11}. In what follows, we are going to apply \ref{7.22} to \ref{7.15} with $\Delta _1=\Delta _0/3$. By \ref{5.4}, we can choose $(\alpha _1,\ell _1)$ which is ltd with respect to $(\Delta ,r,s,K_0)$, with $\ell _1\subset [y_0,y_{\Delta _0}]$, such that $a'(\alpha _1)>c_0$, where $c_0$ depends only on the topological type of $S$. Now if $\alpha \cap \omega _j(\beta )\ne \emptyset $ for some $\beta $ as in Property 5 and $u_j(\beta )\le t_1+\Delta _0/p\le u_{j-1}(\beta )$ then $\alpha \cap \beta '\ne \emptyset $ for some $\beta '\supset \omega _j(\beta )$ with $a'(\beta ')\le e^{-\Delta _0/4}$ and $|\varphi _t(\partial \beta ')|\le e^{-\Delta _0/p}$ for $|t-t_1|\le \Delta _0/p$. Then $\partial \beta '\cap \alpha \ne \emptyset $ if $\nu _0>e^{-\Delta _)/p}$, as we can assume by taking $\Delta _0$ sufficiently large, and $\alpha \subset \beta '$, which is impossible. On the other hand if $u_j(\beta )\ge t_1+\Delta _0/p$ then we see that for any segment $\varphi _{t_1}(\gamma )$ across $\varphi _{t_1}(\omega _j(\beta )\setminus \omega _{j+1}(\beta ))$ we have $|\varphi _{t_1}(\gamma )|_{t_1}\le e^{-\Delta _0/2p}|\zeta _1|$. Hence every segment $\varphi _{t_1}(\gamma )$ is adjacent to a segment $\zeta '\subset \zeta _1\setminus \varphi _{t_1}(\omega _j(\beta )\setminus \omega _{j+1}(\beta ))$ on at least one side with $|\varphi _{t_1}(\gamma )\le \mu _j|\zeta '|_{t_1}$, and hence Property 5 holds for $\zeta _1$, provided that $e^{-\Delta _0/2p}<\mu _j/3$, which is true provided that $\Delta _0$ is large enough given $\mu _j$ (for all $j$) and $p$.  

  Now given $\alpha _i$ and $\zeta _i$, we need to find $\alpha _{i+1}$ and $t_{i+1}$ and $\zeta _{i+1}$. We use the inductively obtained properties of $\zeta _i$, and the bounds this gives on segments $\zeta \subset \zeta _i$ for which $|\chi _t(\zeta )|$ is bounded for some $t>t_i$, as described above. We let $B_j$ be the set of all segments $\gamma $ of  $\omega _j(\beta )\setminus \omega _{j+1}(\beta )$, for all $\beta $,  such that $a'(\beta )\le e^{-\Delta _0/3}$ and $u_0(\beta )\ge t_i$ and $|\varphi _t(\partial \beta )|<e^{-\Delta _0/p}$ for $|t-u_0(\beta )|\le \Delta _0/p$. and $\varphi _{t}(\omega _j(\beta )\setminus \omega _{j+1}(\beta ))\cap \chi _t(\zeta _i)\ne \emptyset $, (for one, hence any, $t$). Let $B_j^1$ be the set of all such segments $\gamma $ such that $\varphi _{t_i}(\gamma )$ is not adjacent to a segment of $\chi _{t_i}(\zeta _i)$ which is outside $\varphi _{t_i}(\omega _j(\beta )\setminus \omega _{j+1}(\beta ))$ and $\mu _i^{-1}$ times longer. Let 
$$B_j^2=B_j\setminus B_j^1,\ \ \ B^1=\cup _jB_j^1,\ \ \ B^2=\cup _jB_j^2,\ \ B=\cup _jB_j= B^1\cup B^2.$$
By the properties of the sets $\omega _j(\beta )\setminus \omega _{j+1}(\beta )$ described in \ref{7.11}, the number of elements of $B_j^1$ is bounded by the number of different $(\beta ,j)$, which, by \ref{7.15}, is $N_j$, where $N_j$ depends only on $\mu _{j'}$ for $j'<j$. So, by choice of $\mu _j$, we can take $N_j\mu _j^{1/2}$ as small as we like. By Property 5, we have, for any $t$,
$$\sum _{\gamma \in B^1}|\varphi _{t}(\gamma )\cap \chi _{t}(\zeta _i)|_{t}\le \sum _j\mu _j^{1/2}N_j\cdot |\chi _{t}(\zeta _i)|_{t}.$$
By the properties of the set $\omega _j(\beta )\setminus \omega _{j+1}(\beta )$ of \ref{7.11} we have 
$$\sum _{\gamma \in B^2}|\varphi _{t}(\gamma )\cap \chi _{t}(\zeta _i)|_{t}\le \sum _j\mu _jN_j\cdot |\chi _{t}(\zeta _i)|_{t}$$
We have
$$\chi _t(\zeta _i)\subset (\cup _{\gamma \in B}\varphi _t(B))\cup (\cup _{\zeta \in A}\chi _t(\zeta ))$$
(for any $t$) and hence, since segments in $A$ have disjoint interiors, we have
$$\sum _{\zeta \in A}|\chi _t(\zeta )|_t\ge |\chi _t(\zeta _i)|_t-\sum _{j\in B_j}|\varphi _t(\gamma )\ge \frac{1}{2} |\chi _t(\zeta _i)|_t.$$
assuming  that $\mu _j$ is sufficiently small given $N_j$, for each $j$, that
$$\sum _jN_j\mu _j^{1/2}<\frac{1}{4}.$$
Now for $A$ as above, let 
$$A'=\{ \zeta \in A,\gamma \in B^1: \chi _{t_i}(\zeta )\cap \varphi _{t_i}(\gamma )\ne \emptyset \} $$
then assuming that $e^{-\Delta _0/4}\le \mu _j^{1/2}$ for all $j$ and $t$,
$$\sum _{ \zeta \in A' }|\chi _{t}(\zeta )|_{t}\le 2 \sum _{\gamma \in B^1}|\varphi _t(\gamma )|_t$$
$$\le 2\sum _jN_j\mu _j^{1/2}\cdot |\chi _t(\zeta _i)|_t.$$
For any $K>0$, if 
$$A^{K,j}=\{ \zeta \in A:|\chi _{t}(\zeta )\cap \varphi _{t}(\gamma )|_{t}\ge K|\chi _{t}(\zeta )|_{t}{\rm{\ for\ some\ }}\gamma \in B_j^2\} $$
then 
$$\sum _{\zeta \in A^{K,j}}|\chi _{t}(\zeta )|_{t}\le K^{-1}\sum _{\zeta \in A^{K,j}}|\chi _{t}(\zeta )\cap (\cup _{\gamma \in B_j^2} \varphi _{t}(\gamma ))|_{t}$$
$$\le K^{-1}\sum _{\gamma \in B_j^2}\varphi _{t}(\gamma )|_{t}\le K^{-1}N_j\mu _j|\chi _t(\zeta )|_t.$$
Now put $A^j=A^{K,j}$ for $K=\mu _j^{1/2}/M$, for $M$ as in \ref{7.17}. We have 
$$\sum _{\zeta \in A^{j}}|\chi _{t}(\zeta )|_{t}\le N_jM\mu _j^{1/2}|\chi _t(\zeta _i)|_t.$$
So 
$$\sum _j\sum _{\zeta \in A^{j}}|\chi _{t}(\zeta )|_{t}+\sum _{ \zeta \in A' }|\chi _{t}(\zeta )|_{t}\le (2+M)\sum _jN_j\mu _j^{1/2}\cdot |\chi _t(\zeta _i)|_t$$
$$\le 2(2+M)\sum _jN_j\mu _j^{1/2}\cdot\sum _{\zeta \in A}|\chi _t(\zeta )|_t.$$
So assuming that $\mu _j$ is big enough given $N_j$ for each $j$, and given $M$, so that
$$ 4(2+M)\sum _jN_j\mu _j^{1/2}<1,$$
we have 
$$A\setminus (A'\cup _jA^j)\ne \emptyset .$$
We choose any $\zeta \in A\setminus (A'\cup _jA^j)$ and, using the definition of $A$, let $(\alpha ,\ell )$ be any ltd with $\ell \subset [y_{t_i+\Delta _0/2},y_{t_i+\Delta _0}]$ and let $\zeta '$ be such that $\chi _t(\zeta ')\subset \chi _t(\zeta )\cap \varphi _t(\alpha )$ for any $t$  and $\chi _{t}(\zeta ' )$ has Poincar\'e length which is boundedly proportional to the injectivity radius for some $y_t=[\varphi _t]\in \ell $. By the definition of segments of $A$, such an $(\alpha ,\ell )$ does exist. Then, since $\zeta \notin A'$,  
$$\chi _t(\zeta ')\cap \varphi _t(\gamma )=\emptyset{\rm{\ for\ all\ }}\gamma \in B^1,$$
 and since $\zeta \notin A^j$ for any $j$, we have, for all $\gamma \in B_j^2$, for any $j$, and any $t$
 $$|\chi _t(\zeta ')\cap\varphi _t(\gamma )|_t\le \sqrt{\mu _j}|\chi _t(\zeta ')|_t.$$
 Then we take $\alpha =\alpha _{i+1}$, $\ell =\ell _{i+1}$ and $\zeta _{i+1}=\zeta '$, and choose $t_{i+1}=t$ so that $y_t\in \ell $ and $\chi _{t}(\zeta ' )$ has Poincar\'e length which is boundedly proportional to the injectivity radius.  Then all the required properties hold for $(\alpha _{i+1},\ell _{i+1})$, $t_{i+1}$ and $\zeta _{i+1}$, including Property 5.
\Box

\end{document}